\documentclass{article}

 \usepackage[utf8]{inputenc} 
 \usepackage[T1]{fontenc}    
 \usepackage{url}            
 \usepackage{booktabs}       
 \usepackage{threeparttable}
 \usepackage{amsfonts}       
 \usepackage{nicefrac}       
 \usepackage{microtype}      
 \usepackage{xcolor}         
 
\usepackage[title]{appendix} 
\usepackage{subeqnarray}
\usepackage{bm}
\usepackage{tikz}
\usepackage{dsfont}
\usepackage{upgreek}
\usepackage{fancyhdr}
\usepackage{multirow}
\usepackage{multicol}
\usepackage{float}
\usepackage[inline]{enumitem}   

\usepackage[top=2.5cm, bottom=2.5cm, left=3cm, right=3cm]{geometry} 

\usepackage[ruled,vlined,linesnumbered,inoutnumbered]{algorithm2e}

\usepackage{bbding}
\usepackage{caption}
\usepackage{todonotes}

\usepackage{mathrsfs}

\usepackage{amsmath}
\usepackage{amsthm}
\usepackage{amssymb}

\usepackage[breaklinks,colorlinks,linkcolor=black,citecolor=black,urlcolor=blue]{hyperref}

\usepackage{graphics,graphicx,epsf,epstopdf,subfigure}
\graphicspath{{./Figures/}}
\ifpdf
\DeclareGraphicsExtensions{.eps,.pdf,.png,.jpg}
\else
\DeclareGraphicsExtensions{.eps}
\fi

\providecommand{\U}[1]{\protect\rule{.1in}{.1in}}

\newtheorem{theorem}{Theorem}[section]

\newtheorem{lemma}[theorem]{Lemma}
\newtheorem{proposition}[theorem]{Proposition}
\newtheorem{definition}[theorem]{Definition}

\newtheorem{assumption}{Assumption}

\numberwithin{equation}{section}

\newcommand{\tr}{\mathrm{tr}}

\newcommand{\st}{\mathrm{s.\,t.}\,\,}

\newcommand{\bR}{\mathbb{R}}
\newcommand{\bN}{\mathbb{N}}
\newcommand{\bE}{\mathbb{E}}

\newcommand{\cA}{\mathcal{A}}
\newcommand{\cC}{\mathcal{C}}
\newcommand{\cL}{\mathcal{L}}
\newcommand{\cB}{\mathcal{B}}
\newcommand{\cR}{\mathcal{R}}

\newcommand{\tG}{\mathtt{G}}
\newcommand{\tV}{\mathtt{V}}
\newcommand{\tE}{\mathtt{E}}

\newcommand{\Rn}{\mathbb{R}^{n}}
  
\newcommand{\Rnp}{\mathbb{R}^{n \times p}} 
 
\newcommand{\Rpp}{\mathbb{R}^{p \times p}} 
\newcommand{\Rnn}{\mathbb{R}^{n \times n}}  
\newcommand{\Rnm}{\mathbb{R}^{n \times m}}
\newcommand{\Rmn}{\mathbb{R}^{m \times n}}

\newcommand{\barX}{\bar{X}}
\newcommand{\barD}{\bar{D}}
\newcommand{\barS}{\bar{S}}
\newcommand{\barE}{\bar{E}}

\newcommand{\barT}{\bar{T}}

\newcommand{\avX}{\bar{\mathbf{X}}}
\newcommand{\avD}{\bar{\mathbf{D}}}
\newcommand{\avS}{\bar{\mathbf{S}}}
\newcommand{\avE}{\bar{\mathbf{E}}}

\newcommand{\avT}{\bar{\mathbf{T}}}

\newcommand{\bfX}{\mathbf{X}}
\newcommand{\bfD}{\mathbf{D}}
\newcommand{\bfS}{\mathbf{S}}
\newcommand{\bfH}{\mathbf{H}}
\newcommand{\bfone}{\mathbf{1}}
\newcommand{\bfLambda}{\mathbf{\Lambda}}
\newcommand{\bfW}{\mathbf{W}}
\newcommand{\bfJ}{\mathbf{J}}

\newcommand{\avXkn}{\bar{\mathbf{X}}^{(k + 1)}}

\newcommand{\bfXkn}{\mathbf{X}^{(k + 1)}}

\newcommand{\barXkn}{\bar{X}^{(k + 1)}}

\newcommand{\zz}{^{\top}}

\newcommand{\ff}{_{\mathrm{F}}}
\newcommand{\fs}{^2_{\mathrm{F}}}

\newcommand{\Snp}{\mathcal{S}_{n,p}}

\newcommand{\dkh}[1]{\left(#1\right)}
\newcommand{\hkh}[1]{\left\{#1\right\}}
\newcommand{\fkh}[1]{\left[#1\right]}
\newcommand{\jkh}[1]{\left\langle#1\right\rangle}
\newcommand{\norm}[1]{\left\|#1\right\|}
\newcommand{\abs}[1]{\left\lvert #1\right\rvert}

\newcommand{\sym}{\mathrm{sym}}

\newcommand{\iid}{i \in [d]}

\newcommand{\iin}{i \in [n]}
\newcommand{\sumiid}{\sum\limits_{i=1}^d}

\allowdisplaybreaks
\graphicspath{ {./results/} }
\definecolor{Gray}{rgb}{0.5,0.5,0.5}

\makeatletter

\newcommand{\Rmnum}[1]{\expandafter\@slowromancap\romannumeral #1@}
\makeatother

\begin{document}

\title{A Variance-Reduced Stochastic Gradient Tracking Algorithm 
	for Decentralized Optimization with Orthogonality Constraints}

\author{
	Lei Wang\thanks{State Key Laboratory of Scientific and Engineering 
		Computing, Academy of Mathematics and Systems Science, Chinese  Academy of 
		Sciences, and University of Chinese Academy of Sciences, China 
		(\href{mailto:wlkings@lsec.cc.ac.cn}{wlkings@lsec.cc.ac.cn}). 
		Research is supported by the National Natural Science Foundation of China 
		(No. 11971466 and 11991020).}
	\and 
	Xin Liu\thanks{State Key Laboratory of Scientific and Engineering 
		Computing, Academy of Mathematics and Systems Science, Chinese Academy of 
		Sciences, 
		and University of Chinese Academy of Sciences, China 
		(\href{mailto:liuxin@lsec.cc.ac.cn}{liuxin@lsec.cc.ac.cn}). 
		Research is supported in part by the National Natural Science Foundation of 
		China (No. 12125108, 11991021, and 12288201), 
		Key Research Program of Frontier Sciences, 
		Chinese Academy of Sciences (No. ZDBS-LY-7022).}
}

\date{} 
\maketitle

\begin{abstract}
	Decentralized optimization with orthogonality constraints
	is found widely in scientific computing and data science.
	Since the orthogonality constraints are nonconvex,
	it is quite challenging to design efficient algorithms.
	Existing approaches leverage the geometric tools from Riemannian optimization
	to solve this problem at the cost of high sample and communication complexities.
	To relieve this difficulty, 
	based on two novel techniques that can waive the orthogonality constraints,
	we propose a variance-reduced stochastic gradient tracking (VRSGT) algorithm
	with the convergence rate of $O(1 / k)$ to a stationary point.
	To the best of our knowledge, 
	VRSGT is the first algorithm for decentralized optimization with orthogonality constraints
	that reduces both sampling and communication complexities simultaneously.
	In the numerical experiments, 
	VRSGT has a promising performance in a real-world autonomous driving application.

\end{abstract}


\section{Introduction}

This paper focuses on the multi-agent optimization problem \eqref{eq:opt-stiefel} 
with orthogonality constraints
that is defined on a connected network $\tG = (\tV, \tE)$
with $\tV = [d] := \{1, 2, \dotsc, d\}$ being composed of all the agents
and $\tE = \{(i, j) \mid i < j, i \text{~and~} j \text{~are connected}\}$
being the set of all the communication links.
\begin{equation}\label{eq:opt-stiefel}
	\begin{aligned}
		\min\limits_{X \in \Rnp} \hspace{2mm} 
		& f(X) := \dfrac{1}{d} \sumiid f_i(X) \\
		\st \hspace{3mm} & X\zz X = I_p.
	\end{aligned}
\end{equation}
Here, $I_p$ stands for the $p \times p$ identity matrix,
and $f_i: \Rnp \to \bR$ is a local loss function privately known by the $i$-th agent for $i \in [d]$.
The set of all $n \times p$ orthogonal matrices 
is referred to as Stiefel manifold \cite{Stiefel1935}
denoted by $\Snp = \{X \in \Rnp \mid X\zz X = I_p\}$,
which is a special case of Riemannian manifold \cite{Absil2008} embedded in $\Rnp$.

Problems of the form \eqref{eq:opt-stiefel}
have profound applications in various scientific and engineering areas,
such as spectral analysis \cite{Kempe2008,Huang2020},
dictionary learning \cite{Raja2015},
eigenvalue estimation of the covariance matrix \cite{Penna2014},
and deep neural networks with orthogonal constraints 
\cite{Arjovsky2016,Vorontsov2017,Huang2018}.
It is noteworthy that, in these applications,
each local function $f_i$ is an average cost of different samples.
Without loss of generality, we assume that each agent owns $l$ samples
and $f_i$ can be represented as follows.
\begin{equation} \label{eq:finite-sum}
	f_i (X) = \dfrac{1}{l} \sum_{j = 1}^{l} f^{[j]}_{i} (X),
\end{equation}
where $f^{[j]}_{i}$ denotes the cost for the $j$-th data sample at the $i$-th agent.

\subsection{Decentralized Formulation}

In this paper, we consider the scenario that
only local communications are permissible,
namely, the agents can only exchange information 
with their immediate neighbors through the network.
Thus, the global information is not available,
and the agent $i$ needs to maintain a local copy $X_i$ of the common variable $X$.
Each agent only uses the local information to update its local variable
and then communicates with neighbors to reach a consensus.
This type of problem is usually called {\it decentralized} optimization.

For convenience, we use the notation 
$\bfX = [X_1\zz, \dotsc, X_d\zz]\zz \in \bR^{dn \times p}$ 
to stack all the local variables.
Now the concerned problem \eqref{eq:opt-stiefel} can be recast as 
the following decentralized formulation.
\begin{subequations}\label{eq:opt-dest}
	\begin{align}
		\label{eq:obj-dest}
		\min\limits_{ \bfX } \hspace{4mm} 
		& \dfrac{1}{d} \sumiid f_i(X_i) \\
		\label{eq:con-consensus}
		\st \hspace{2.5mm} &  X_i = X_j, \hspace{2mm} (i, j) \in \tE, \\
		\label{eq:con-orth}
		& X_i\zz X_i = I_p, \hspace{2mm} \iid.
	\end{align}
\end{subequations}

Since the network $\tG$ is assumed to be connected,
the constraints \eqref{eq:con-consensus} enforce the consensus condition
$X_1 = X_2 = \dotsb = X_d$.
Therefore, our goal is to seek a consensus 
such that each local variable $X_i$ is a first-order stationary point of \eqref{eq:opt-stiefel}
whose definition can be stated as follows.

\begin{definition}[\cite{Wang2021multipliers}]
	A point $X \in \Rnp$ is called a first-order stationary point of \eqref{eq:opt-stiefel}
	if and only if 
	\begin{equation*}
		0 = \mathrm{grad}\, f(X) 
		:= \nabla f(X) - X \sym \dkh{X\zz \nabla f(X)}
		\mbox{~and~}
		X\zz X = I_p.
	\end{equation*}
	Here, $\mathrm{grad}\, f(X)$ denotes 
	the Riemannian gradient \cite{Absil2008} of $f$ at $X$,
	and $\sym (B) = (B + B\zz) / 2$ represents 
	the symmetric part of the given matrix $B$.
\end{definition}


In the decentralized optimization, 
the structure of the communication network $\tG$
is described by a mixing matrix $W = [W(i, j)] \in \bR^{d \times d}$.
We assume that $W$ satisfies the following properties,
which are standard in the literature.
There are a few common choices for the mixing matrix $W$.
Interested readers are referred to \cite{Xiao2004,Shi2015} for more details.

\begin{assumption} \label{asp:network}
	The mixing matrix $W$ is symmetric, nonnegative, 
	and doubly stochastic (i.e., $W \mathbf{1}_d = W\zz \mathbf{1}_d = \mathbf{1}_d$).
	Moreover, $W(i, j) = 0$ if $i \neq j$ and $(i, j) \notin \tE$.
\end{assumption}


\subsection{Existing Works}

\label{subsec:existing}

In this subsection, we briefly review some existing approaches 
closely related to the present work.

\paragraph{Riemannian SVRG methods.}
Recent years have witnessed the rapid development of optimization algorithms
on the Riemannian manifolds.
Interested readers are referred to a comprehensive survey \cite{Hu2020brief} for more details.
In the following, we briefly review several Riemannian SVRG algorithms.
Based on the geometric tools developed in \cite{Absil2008},
a large amount of efforts has been devoted to extending the SVRG-type algorithm
from the Euclidean space to the Riemannian manifolds.
Existing SVRG algorithms for manifold optimization can be divided into two categories.
The first category constructs stochastic variance reduced Riemannian gradients
by invoking parallel transports or vector transports \cite{Absil2008},
including \cite{Zhang2016riemannian,Sato2019riemannian}.
As illustrated in \cite{Qi2010riemannian},
the evaluation of parallel transports or vector transports
is computationally expensive.
Therefore, the second category of Riemannian SVRG algorithms, 
such as \cite{Jiang2017vector,Liu2019b},
is proposed to waive the computation of these two operations.
However, the aforementioned methods 
can hardly be extended to the decentralized setting.

\paragraph{Decentralized algorithms in the Euclidean space.}
Decentralized optimization has been adequately investigated in the Euclidean space
with different types of decentralized algorithms emerging in large numbers,
such as decentralized gradient descent (DGD) \cite{Nedic2009,Yuan2016,Zeng2018},
decentralized gradient tracking (DGT) \cite{Shi2015,Xu2015,Qu2017},
decentralized stochastic gradient descent (DSGD) \cite{Lian2017can},
decentralized SVRG method \cite{Sun2020improving,Xin2022fast},
and decentralized primal-dual algorithm
\cite{Chang2015multi,Ling2015,Hong2017,Hajinezhad2019}.
It is worthy of mentioning that none of the above algorithms 
are able to deal with nonconvex constraints,
and thus not applicable to the problem \eqref{eq:opt-dest}.
We refer to three recent survey papers  
\cite{Nedic2018network,Xin2020general,Chang2020distributed}
for a complete review on the decentralized algorithms in the Euclidean space.

\paragraph{Decentralized Riemannian gradient descent methods.}
$ $
To solve decentralized problems over the Riemannian manifold, 
such as problem \eqref{eq:opt-dest},
Chen et al. \cite{Chen2021decentralized} propose two algorithms DRGTA and DRSGD,
which are the Riemannian extensions of DGT and DSGD, respectively.
Although the exact convergence is guaranteed with fixed stepsizes,
full batch gradient computation is involved in each iteration of DRGTA,
giving rise to high sample complexity.
On the contrary, DRSGD requires diminishing stepsizes to find an exact solution,
which significantly slows down the convergence.
This incurs high communication complexity,
since it is necessary for DRSGD to take numerous rounds of communications
to reach certain accuracy.
Finally, we note that both DRGTA and DRSGD require 
$t \geq \lceil O(- \log_{\lambda} (2 \sqrt{d})) \rceil$ rounds of communications 
per iteration to guarantee the convergence.
For sparse networks, $t$ is in the order of $O(d^2 \log d)$ \cite{Chen2021decentralized}.
This prerequisite further imposes an enormous burden on the communications.

\subsection{Our Contributions}



The main contributions of this paper are three folds.

\begin{enumerate}
	
	\item We propose a variance-reduced stochastic gradient tracking (VRSGT) algorithm
	for solving problem \eqref{eq:opt-dest} to reduce both sampling and communication complexities
	at the same time.
	Please see Table \ref{tb:complexity} for detailed comparisons with existing peer methods. 
	
	\item VRSGT is based on two novel techniques, 
	augmented Lagrangian estimation and gradient approximation,
	to deal with orthogonality constraints, 
	which are of independent values. 
	Basically, one can utilize any unconstrained algorithms 
	to solve the problem \eqref{eq:opt-dest} 
	by resorting to these two techniques.
	The convergence of VRSGT is established in  
	Theorem \ref{thm:convergence}.
	
	\item We apply VRSGT to solve dual principal component pursuit (DPCP) problems
	that play an important role
	in autonomous driving.
	VRSGT has the ability to recognize the objects in the road quickly and precisely. 
	
\end{enumerate}

\begin{table}[ht!]
	\centering
	\caption{A comparison of sample and communication complexities 
		to find a first-order $\epsilon$-stationary point (to de defined in Definition \ref{def:epsilon})
		on decentralized algorithms for the problem \eqref{eq:opt-dest}.}
	\label{tb:complexity} 
	\begin{threeparttable}
		\begin{tabular}{c|c|c} 
			\toprule 
			Algorithm & Sample complexity & Communication complexity \\
			\midrule
			DRGTA \cite{Chen2021decentralized}  &  $O(dl / \epsilon)$  
			&  $O(t / \epsilon)$\tnote{1} \\
			\midrule
			DRSGD \cite{Chen2021decentralized}  &  $O(d / \epsilon^2)$  
			&  $O(t / \epsilon^2)$\tnote{1} \\
			\midrule
			VRSGT (this work)  &  $O(d\sqrt{l} / \epsilon)$  
			&  $O(1 / \epsilon)$ \\
			\bottomrule 
		\end{tabular} 
		\begin{tablenotes}
			\item [1] $t \geq \lceil O(- \log_{\lambda} (2 \sqrt{d})) \rceil$.
			For sparse networks, $t$ can be $O(d^2 \log (d)$ \cite{Chen2021decentralized}.
		\end{tablenotes}
	\end{threeparttable}
\end{table}

	
\subsection{Notations}

The Euclidean inner product of two matrices \(Y_1, Y_2\) 
with the same size is defined as \(\jkh{Y_1, Y_2}=\tr(Y_1\zz Y_2)\),
where $\tr (B)$ stands for the trace of a square matrix $B$.
And the notation $\sym (B) = (B + B\zz) / 2$ represents 
the symmetric part of $B$.
The Frobenius norm and 2-norm of a given matrix \(C\) 
is denoted by \(\norm{C}\ff\) and \(\norm{C}_2\), respectively. 
The $(i, j)$-th entry of a matrix $C$ is represented by $C(i, j)$.
The notation $\bfone_d$ stands for the $d$-dimensional vector of all ones.
The Kronecker product is denoted by $\otimes$.
Given a differentiable function \(g(X) : \Rnp \to \bR\), 
the gradient of \(g\) with respect to \(X\) is represented by \(\nabla g(X)\).
Further notation will be introduced wherever it occurs.

\section{Algorithm Design}

In this section, we develop two novel techniques 
to deal with the nonconvex orthogonality constraints.
Based on these techniques, 
we propose a variance-reduced stochastic gradient tracking algorithm. 

\subsection{Augmented Lagrangian Estimation}

\label{subsec:augmented}

The augmented Lagrangian function of \eqref{eq:opt-dest} can be represented as follows.
\begin{equation*}
	\cL(\bfX, \bfLambda) := \dfrac{1}{d} \sumiid \cL_i (X_i, \Lambda_i),
\end{equation*}
where
\begin{equation*}
	\cL_i (X_i, \Lambda_i)  := 
	f_i (X_i) - \dfrac{1}{2} \jkh{\Lambda_i, X_i\zz X_i - I_p} 
	+ \dfrac{\beta}{4} \norm{X_i\zz X_i - I_p}\fs.
\end{equation*}
In the above augmented Lagrangian function,
we only penalize the orthogonality constraints \eqref{eq:con-orth}.
Here, $\beta > 0$ is a penalty parameter,
$\Lambda_i \in \Rpp$ denotes the local Lagrangian multiplier 
corresponding to the orthogonality constraint of $X_i$,
and $\bfLambda = [\Lambda_1\zz, \dotsc, \Lambda_d\zz]\zz \in \bR^{dp \times p}$
stacks all the local multipliers.

%


According to the discussions in \cite{Gao2019,Xiao2020,Wang2021multipliers},
the multiplier associated with orthogonality constraints admits a closed-form expression
at any first-order stationary point.
Therefore, it is straightforward to verify that 
each local multiplier admits the following explicit expression.
\begin{equation*}
	\Lambda_i = \sym \dkh{X_i\zz \nabla f_i (X_i)}, \iid.
\end{equation*}
The second term of $\cL_i (X_i, \Lambda_i)$ with the above expression of $\Lambda_i$
can be rewritten as the following form.
\begin{equation*}
	\jkh{\Lambda_i, X_i\zz X_i - I_p} = \jkh{\nabla f_i (X_i), X_i X_i\zz X_i - X_i}.
\end{equation*}
Then in light of the definition of directional derivative, we have
\begin{equation}\label{eq:approximate}
	\lim_{\sigma \to 0} \dfrac{f_i (X_i + \sigma (X_i X_i\zz X_i - X_i)) - f_i (X_i)}{\sigma}
	= \jkh{\nabla f_i (X_i), X_i X_i\zz X_i - X_i}.
\end{equation}
For convenience, we denote
\begin{equation} \label{eq:y_sigma}
	\cA (X_i) := X_i + \sigma (X_i X_i\zz X_i - X_i)
	= (1 - \sigma) X_i + \sigma X_i X_i\zz X_i.
\end{equation}
This motivates us to replace the second term of $\cL_i (X_i, \Lambda_i)$ 
with $(f_i (\cA (X_i)) - f_i (X_i)) / \sigma$
for a small constant $\sigma \neq 0$,
and the following function is obtained.
\begin{equation*}
	h_{i} (X_i) := g_{i} (X_i) + \beta b (X_i),
\end{equation*}
where
\begin{equation*}
	g_{i} (X) := \dkh{1 + \dfrac{1}{2 \sigma}} f_i (X_i) 
	- \dfrac{1}{2 \sigma} f_i (\cA (X_i)),
	\mbox{~and~}
	b (X_i) := \dfrac{1}{4} \norm{X_i\zz X_i - I_p}\fs.
\end{equation*}
Finally, the augmented Lagrangian function $\cL(\bfX, \bfLambda)$ 
can be estimated by the following function.
\begin{equation} \label{eq:ALE}
	h (\bfX) = \dfrac{1}{d} \sumiid h_{i} (X_i),
\end{equation}
which is called {\it augmented Lagrangian estimation}.

\subsection{Gradient Approximation}

\label{subsec:gradient}

Now we discuss the computation related to the gradient of 
the augmented Lagrangian estimation $h (\bfX)$.
After straightforward calculations, we have
\begin{equation*}
	\nabla h (\bfX) =
	[\nabla h_{1} (X_1)\zz, \dotsc, \nabla h_{d} (X_d)\zz]\zz,
\end{equation*}
where
\begin{equation} \label{eq:gradient}
	\begin{aligned}
		\nabla h_{i} (X_i) 
		= & \dkh{1 + \dfrac{1}{2 \sigma}} \nabla f_i (Z) |_{Z = X_i} 
		- \nabla f_i (Z)|_{Z = \cA (X_i)} 
		\dkh{\dfrac{1 - \sigma}{2 \sigma} I_p + \dfrac{1}{2} X_i\zz X_i}\\
		& - X_i \sym \dkh{ X_i\zz \nabla f_i (Z) |_{Z = \cA (X_i)} }
		+ \beta X_i (X_i\zz X_i - I_p).
	\end{aligned}
\end{equation}
The gradient of $f_i$ is evaluated twice
at different points $X_i$ and $\cA (X_i)$ in $\nabla h_{i} (X_i) $,
which incurs high computational costs.
We find that these two points are very close 
as long as the orthogonality violation of $X_i$ is small 
since $\norm{\cA (X_i) - X_i}\ff 
\leq \sigma \norm{X_i}_2 \norm{X_i\zz X_i - I_p}\ff$.
Therefore, we propose to approximate $\nabla f_i (Z) |_{Z = \cA (X_i)}$
by $\nabla f_i (Z) |_{Z = X_i} $ to reduce the computational costs.
This yields the following direction $H_{i} (X_i) \in \Rnp$
as an approximation of $\nabla h_{i} (X_i)$.
\begin{equation} \label{eq:Hi}
		H_{i} (X_i) := G_{i} (X_i) + \beta E (X_i),
\end{equation}
where $E (X_i) := X_i (X_i\zz X_i - I_p)$ and
\begin{equation} \label{eq:Gi}
	G_i (X_i) := \nabla f_i (Z)|_{Z = X_i} 
	\dkh{\dfrac{3}{2} I_p - \dfrac{1}{2} X_i\zz X_i}
	- X_i \sym \dkh{ X_i\zz \nabla f_i (Z) |_{Z = X_i} } 
\end{equation}
Let $H (X) := \sum_{i = 1}^d H_i (X) / d$ and $G (X) := \sum_{i = 1}^d G_i (X) / d$.
We have the following lemma for $H (X)$ and $G (X)$.

\begin{lemma} \label{le:expen}
	Let $\cR := \{X \in \Rnp \mid \|X\zz X - I_p\|\ff \leq 1 / 6\}$
	be a bounded region 
	and $M := \sup \{ \| \nabla f (X) \|\ff  
	\mid X \in \cR\}$ 
	be a positive constant.
	Then if $\beta \geq (6 + 21 M) / 5$, we have
	\begin{equation*}
		\norm{H (X)}\fs \geq \norm{G (X)}\fs 
		+ \beta \norm{X\zz X - I_p}\fs,
	\end{equation*}
	for any $X \in \cR$.
\end{lemma}


The proof of Lemma \ref{le:expen} is arranged to Appendix \ref{apx:expen}. 
It is noteworthy that $G (X)  = \mathrm{grad}\, f (X)$ 
for any $X \in \Snp$.
Hence, $G (\cdot)$ can be interpreted as 
the generalized Riemannian gradient of $f$ in the Euclidean space.
Furthermore, Lemma \ref{le:expen} shows that
both the orthogonality violation of $X$ and the norm of $G (X)$
is controlled by the norm of $H (X)$
as long as $X$ is restricted in the bounded region $\Omega$.
Consequently, under the decentralized setting, 
$H_{i} (\cdot) $ can be employed as a local descent direction
since the substationarity and orthogonality violation can be reduced simultaneously.
In addition, for the finite sum setting \eqref{eq:finite-sum} considered in this paper, we have
\begin{equation*}
	H_{i} (X_i) = \dfrac{1}{l} \sum_{j = 1}^l H^{[j]}_{i} (X_i),
\end{equation*}
where $H^{[j]}_{i} (X_i) 
= \nabla f^{[j]}_{i} (Z)|_{Z = X_i} ( 3 I_p - X_i\zz X_i ) / 2
- X_i \sym ( X_i\zz \nabla f^{[j]}_{i} (Z) |_{Z = X_i} ) + \beta E (X_i)$.

Finally, under the decentralized setting, 
we introduce the following stationarity gap $\mathtt{StaGap} (\bfX)$ for \eqref{eq:opt-dest}.
\begin{equation} \label{eq:optgap}
	\mathtt{StaGap} (\bfX) 
	:= {} \norm{\dfrac{1}{d} \sumiid G_i (\barX)}\fs
	+ \dfrac{1}{d} \sumiid \norm{X_i - \barX}\fs 
	+ \norm{\barX\zz \barX - I_p}\fs,
\end{equation}
where $\barX := \sum_{r = 1}^d X_r / d$.

Based on the stationarity gap, we define the following 
first-order $\epsilon$-stationary point of \eqref{eq:opt-dest}.

\begin{definition} \label{def:epsilon}
	For a sufficiently small constant $\epsilon > 0$,
	a point $\bfX = [X_1\zz, \dotsc, X_d\zz]\zz$ 
	is called a first-order $\epsilon$-stationary point of \eqref{eq:opt-dest} 
	if and only if $\mathtt{StaGap} (\bfX) \leq \epsilon$.
\end{definition}

\subsection{Algorithmic Development}

Based on the augmented Lagrangian estimation \eqref{eq:ALE}
developed in Subsection \ref{subsec:augmented},
\eqref{eq:opt-dest} can be converted to the following unconstrained problem.
\begin{equation}\label{eq:opt-ale}
	\begin{aligned}
		\min\limits_{ \bfX } \hspace{4mm} & h (\bfX) \\
		\st \hspace{2.5mm} &  X_i = X_j, \hspace{2mm} (i, j) \in \tE, \\
	\end{aligned}
\end{equation}
Now we can design decentralized algorithms to solve \eqref{eq:opt-ale}
without consideration of orthogonality constraints.
The gradient approximation techniques developed in Subsection \ref{subsec:gradient}
can be applied to reduce the computational costs.

To introduce our algorithm, 
we first define two auxiliary local variables $S_i \in \Rnp$ and $D_i \in \Rnp$.
Specifically, $S_i$ aims at estimating the local full batch descent direction
$\sum_{j = 1}^l H^{[j]}_{i} (X_i) / l$ by only using a subset of samples,
while $D_i$ is designed to track the global descent direction 
$\sum_{r = 1}^d H_{r} (X_r) / d$ based on $S_i$.
After the local and global descent direction estimates are obtained,
our algorithm performs a local update along the direction of $- D_i$.

In the following discussions,  
we use $X^{(k, t)}_i$ to denote the $t$-th inner iterate of $X_i$ at the $k$-th outer iteration.
Other notations $S^{(k, t)}_i$ and $D^{(k, t)}_i$ are similar.
Please see the main steps of our algorithm below.

\paragraph{Step 1: $\bfX$-update.}
At the $k$-th outer iteration, each agent $i$ first performs a descent 
in the direction of the global descent direction estimate $D^{(k, 0)}_i$,
and then combines the results with its neighbors.
\begin{equation*}
	X^{(k + 1, 1)}_i = \sum_{r = 1}^d W(i, r) \dkh{X^{(k, 0)}_r 
		- \eta D^{(k, 0)}_r},
\end{equation*}
where $\eta > 0$ is the stepsize.
The inner iteration is similar.

\paragraph{Step 2: $\bfS$-update.}
At the $k$-th outer iteration,
each agent $i$ directly computes the local full batch descent direction as follows.
\begin{equation*}
	S^{(k + 1, 1)}_i = H_{i} (X^{(k + 1, 1)}_i).
\end{equation*}
In contrast, the $t$-th inner iteration estimates it 
by using a subset $\cC^{(k+ 1, t)}$ of $\tau$ samples.
\begin{equation*}
	S^{(k + 1, t + 1)}_i = \dfrac{1}{\tau} \sum_{j \in \cC^{(k + 1, t)}} 
	\dkh{H^{[j]}_{i} (X^{(k + 1, t + 1)}_i) - H^{[j]}_{i} (X^{(k + 1, t)}_i)} 
	+ S^{(k + 1, t)}_i,
\end{equation*}
where $\tau \in [1, l]$ is the batch size.

\paragraph{Step 3: $\bfD$-update.}
At the $k$-th outer iteration,
each agent $i$ combines the previous estimate $D^{(k, 0)}_i$ with its neighbors,
and then make a new estimation based on the fresh information $S^{(k + 1, 1)}_i$.
\begin{equation*}
	D^{(k + 1, 1)}_i = \sum_{r = 1}^d W(i, r) D^{(k, 0)}_r 
	+ S^{(k + 1, 1)}_i - S^{(k, 0)}_i.
\end{equation*}
The inner iteration is similar.

We summarize the proposed algorithm as a compact form in Algorithm \ref{alg:VRSGT},
which is called variance-reduced stochastic gradient tracking
and abbreviated to VRSGT.
In Algorithm \ref{alg:VRSGT},
we use $\bfX^{(k, t)} \in \bR^{dn \times p}$, $\bfS^{(k, t)}  \in \bR^{dn \times p}$, 
$\bfD^{(k, t)} \in \bR^{dn \times p}$,
$\bfH (\bfX^{(k, t)}) \in \bR^{dn \times p}$, 
and $\bfH^{[j]} (\bfX^{(k, t)}) \in \bR^{dn \times p}$ to denote
the concatenation of $X^{(k, t)}_i \in \Rnp$, $S^{(k, t)}_i \in \Rnp$, 
$D^{(k, t)}_i \in \Rnp$, $H_{i} (X^{(k, t)}_i)\in \Rnp$,
and $H^{[j]}_{i} (X^{(k, t)}_i) \in \Rnp$ across all the agents, respectively.
For convenience, we denote $\bfW = W \otimes I_n \in \bR^{dn \times dn}$.

\begin{algorithm}[ht!]
	\caption{Variance-Reduced Stochastic Gradient Tracking (VRSGT).} 
	\label{alg:VRSGT}
	
	\KwIn{initial guess $X_{\mathrm{init}}$, 
		stepsize $\eta > 0$, batch size $\tau \in [1, l]$,
		inner iteration interval $q > 0$,
		and penalty parameter $\beta > 0$.}
		
	Set $k := 0$.
	
	Initialize $\bfX^{(0, 0)} = (\bfone_d \otimes I_n) X_{\mathrm{init}}$,
	$\bfS^{(0, 0)} := \bfH (\bfX^{(0, 0)})$ 
	and $\bfD^{(0, 0)} := \bfH (\bfX^{(0, 0)})$. 
	
	\While{``not converged"}
	{
		
		$\bfX^{(k + 1, 1)} :=  \bfW (\bfX^{(k, 0)} - \eta \bfD^{(k, 0)})$.
		
		$\bfS^{(k + 1, 1)} := \bfH (\bfX^{(k + 1, 1)})$.
			
		$\bfD^{(k + 1, 1)} :=  \bfW \bfD^{(k, 0)} + \bfS^{(k + 1, 1)} - \bfS^{(k, 0)}$.
		
		\For{$t = 1, 2, \dotsc, q$}
		{
		
			$\bfX^{(k + 1, t + 1)} :=  \bfW (\bfX^{(k + 1, t)} - \eta \bfD^{(k + 1, t)})$.
	
			Each node draws a subset $\cC^{(k + 1, t)}$ of $\tau$ samples 
			from $[l]$ with replacement.
			
			$\bfS^{(k + 1, t + 1)} := \sum_{j \in \cC^{(k + 1, t)}} 
			\dkh{ \bfH^{[j]} (\bfX^{(k + 1, t + 1)}) 
			- \bfH^{[j]} (\bfX^{(k + 1, t)}) } / \tau 
			+ \bfS^{(k + 1, t)}$.
			
			$\bfD^{(k + 1, t + 1)} :=  \bfW \bfD^{(k + 1, t)} 
			+ \bfS^{(k + 1, t + 1)} - \bfS^{(k + 1, t)}$.
			
		}
	
		Set $\bfX^{(k + 1, 0)} := \bfX^{(k + 1, q + 1)}$, 
		$\bfS^{(k + 1, 0)} := \bfS^{(k + 1, q + 1)}$, 
		and $\bfD^{(k + 1, 0)} := \bfD^{(k + 1, q + 1)}$.
		
		Set $k := k + 1$.
		
	}

	\KwOut{$\bfX^{(k, 0)}$.}
	
\end{algorithm}

The total communication rounds required by VRSGT 
is in the same order as the total number of iterations (both inner and outer iterations), 
since only two rounds of communications are performed per iteration,
via broadcasting the local variables $X_i$ and $D_i$ to the neighbors.
And the total number of samples used per iteration 
is either $d \tau$ (inner iteration) or $dl$ (outer iteration).

\section{Convergence Analysis}

\label{sec:convergence}

Throughout this paper, 
we make the following blanket assumptions about local cost functions.
\begin{assumption} \label{asp:smooth}
	Each local cost function $f^{[j]}_{i}$ is first-order differentiable. 
	And $\{ f^{[j]}_{i} \}_{j = 1}^l$ satisfies the following 
	mean-squared local Lipschitz smoothness property with a constant $L > 0$.
	\begin{equation} \label{eq:Lipschitz}
		\sqrt{\dfrac{1}{l} \sum_{j = 1}^l 
		\norm{\nabla f^{[j]}_{i} (X) - \nabla f^{[j]}_{i} (Z)}\fs} 
		\leq L \norm{X - Z}\ff, \quad
		\mbox{~for any~} i \in [d]
		\mbox{~and~} X, Z \in \tilde{\mathcal{S}},
	\end{equation}
	where $\tilde{\mathcal{S}}$ is an arbitrary bounded set containing $\Snp$.
\end{assumption}

Assumption \ref{asp:smooth} is standard 
in the decentralized stochastic optimization literature,
which is also used in \cite{Sun2020improving,Xin2022fast}.
It should be noted that, in these works,
the condition \eqref{eq:Lipschitz} is assumed to hold for any $X, Z \in \Rnp$.
Since the Stiefel manifold is compact, 
this {\it global} Lipschitz smoothness property is too stringent for \eqref{eq:opt-dest}.
Therefore, in Assumption \ref{asp:smooth},
we only impose the {\it local} version of this property.

Now we present the global convergence of VRSGT as follows,
whose proof is given in Appendix \ref{apx:convergence}.

\begin{theorem} \label{thm:convergence}
	Suppose Assumptions \ref{asp:network} and \ref{asp:smooth} hold.
	Let $\{\bfX^{(k, t)}\}$ be the iterate sequence generated by Algorithm \ref{alg:VRSGT} 
	with $X_{\mathrm{init}} \in \cR$, $\eta \in (0, \bar{\eta})$, $\beta > \underline{\beta}$,
	and $\tau = q = \sqrt{l}$.
	The constants $\bar{\eta}$ and $\underline{\beta}$ 
	are given in Appendix \ref{apx:convergence}.
	Then it holds that
	\begin{equation*}
		\min_{\substack{k = 0, 1, \dotsc, K \\ t = 0, 1, \dotsc, q}} 
		\bE \fkh{ \mathtt{StaGap} (\bfX^{(k, t)}) } 
		\leq \dfrac{C}{\sqrt{l} K},
	\end{equation*}
	where $C > 0$ is a constant defined in Appendix \ref{apx:convergence}
	and the expectation is taken over the randomness from the sampling step.
\end{theorem}

Theorem \ref{thm:convergence} demonstrate that, 
in order to achieve a first-order $\epsilon$-stationary point of \eqref{eq:opt-dest} by VRSGT,
the total number of communication rounds is in the order of $O (1 / \epsilon)$,
and the total number of samples evaluated across the whole network is in the order of $O (d \sqrt{l} / \epsilon)$.

\section{Numerical Experiments}

\label{sec:experiments}

In this section, the numerical performance of VRSGT is evaluated 
through comprehensive numerical experiments in comparison with 
state-of-the-art algorithms.
The corresponding experiments are performed on a workstation
with two Intel Xeon Gold 6242R CPU processors (at $3.10$GHz$\times 20 \times 2$) 
and 510GB of RAM under Ubuntu 20.04.
All the tested algorithms are implemented in the \texttt{Python} language,
and the communication is realized via the package \texttt{mpi4py}.
The code of DRSGD \cite{Chen2021decentralized} is downloaded from
GitHub\footnote{\url{https://github.com/chenshixiang/Decentralized_Riemannian_gradient_descent_on_Stiefel_manifold}}.



\subsection{Test Problems}

\label{subsec:test-problems}

In the numerical experiments, 
we test two different types of problems to evaluate the performances of compared algorithms,
including decentralized (principal component analysis) PCA problem 
and dual principal component pursuit (DPCP) problem.

\paragraph{Decentralized PCA Problem.}
As a fundamental tool for dimensionality reduction,
PCA is usually a critical procedure or preprocessing step 
in a large number of numerous statistical and machine learning tasks.
under the decentralized setting, PCA amounts to solving the following optimization problem.
\begin{equation}\label{eq:opt-pca}
	\begin{aligned}
		\min\limits_{X \in \Rnp} \hspace{2mm} 
		& -\frac{1}{2d} \sumiid \dkh{
			\dfrac{1}{l}\sum_{j = 1}^ l \tr \dkh{ X\zz a_{i (j)} a_{i (j)}\zz X} }\\
		\st \hspace{3mm} & X \in \Snp,
	\end{aligned}
\end{equation}
where $a_{i (j)} \in \Rn$ is the $j$-th sample owned by the $i$-th agent.  
For convenience, 
we denote $A_i = [a_{i (1)}, a_{i (2)}, \dotsc, a_{i (l)}] \in \bR^{n \times l}$
as the local data matrix
and $A = [A_1, A_2, \dotsc, A_d] \in \Rnm$
as the global data matrix with $m = dl$.

\paragraph{Decentralized DPCP Problem.}
DPCP \cite{Zhu2018dual,Ding2019noisy} 
is a recently proposed model for learning subspaces of high relative dimension 
from datasets contaminated by outliers,
which has wide applications in the computer vision, 
such as detecting planar structures in 3D point clouds 
in KITTI dataset \cite{Geiger2013vision} 
and estimating relative poses in multiple-view geometry \cite{Hartley2003multiple}.
Specifically, DPCP aims to recover an $n - 1$ dimensional hyperplane in $\Rn$
by a normal vector 
from $m$ samples stored in the matrix $B \in \Rnm$ to distinguish inliers and outliers.
Under the decentralized setting, we assume that $B = [B_1, B_2, \dotsc, B_d]$
with $B_i = [b_{i (1)}, b_{i (2)}, \dotsc, b_{i (l)}]$,
where $B_i \in \bR^{n \times l}$ contains all the $l$ samples owned by the $i$-th agent,
and $b_{i (j)} \in \Rn$ is the $j$-th sample of $B_i$.
Recently, the following optimization problem with orthogonality constraints 
is proposed in \cite{Hu2020efficient} for DPCP.
\begin{equation}\label{eq:opt-dpcp}
	\begin{aligned}
		\min\limits_{X \in \Rnp} \hspace{2mm} 
		& -\frac{1}{3d} \sumiid \dkh{
			\dfrac{1}{l} \sum_{j = 1}^l \norm{X\zz b_{i (j)}}_3^3 }\\
		\st \hspace{3mm} & X \in \Snp,
	\end{aligned}
\end{equation}
where the $\ell_3$-norm is defined as $\norm{Y}_3^3 = \sum_{i, j} Y(i, j)^3$.

\subsection{Robustness to Penalty Parameter}

\label{subsec:default}

In this subsection, we show that VRSGT is not sensitive to 
the choices of penalty parameter $\beta$.
The following experiments are performed on the Erdos-Renyi network 
with a fixed probability $\mathtt{prob} = 0.5$, 
and the batch size is set to $10$.

\begin{figure}[ht!]
	\centering
	
	\includegraphics[width=0.3\linewidth]{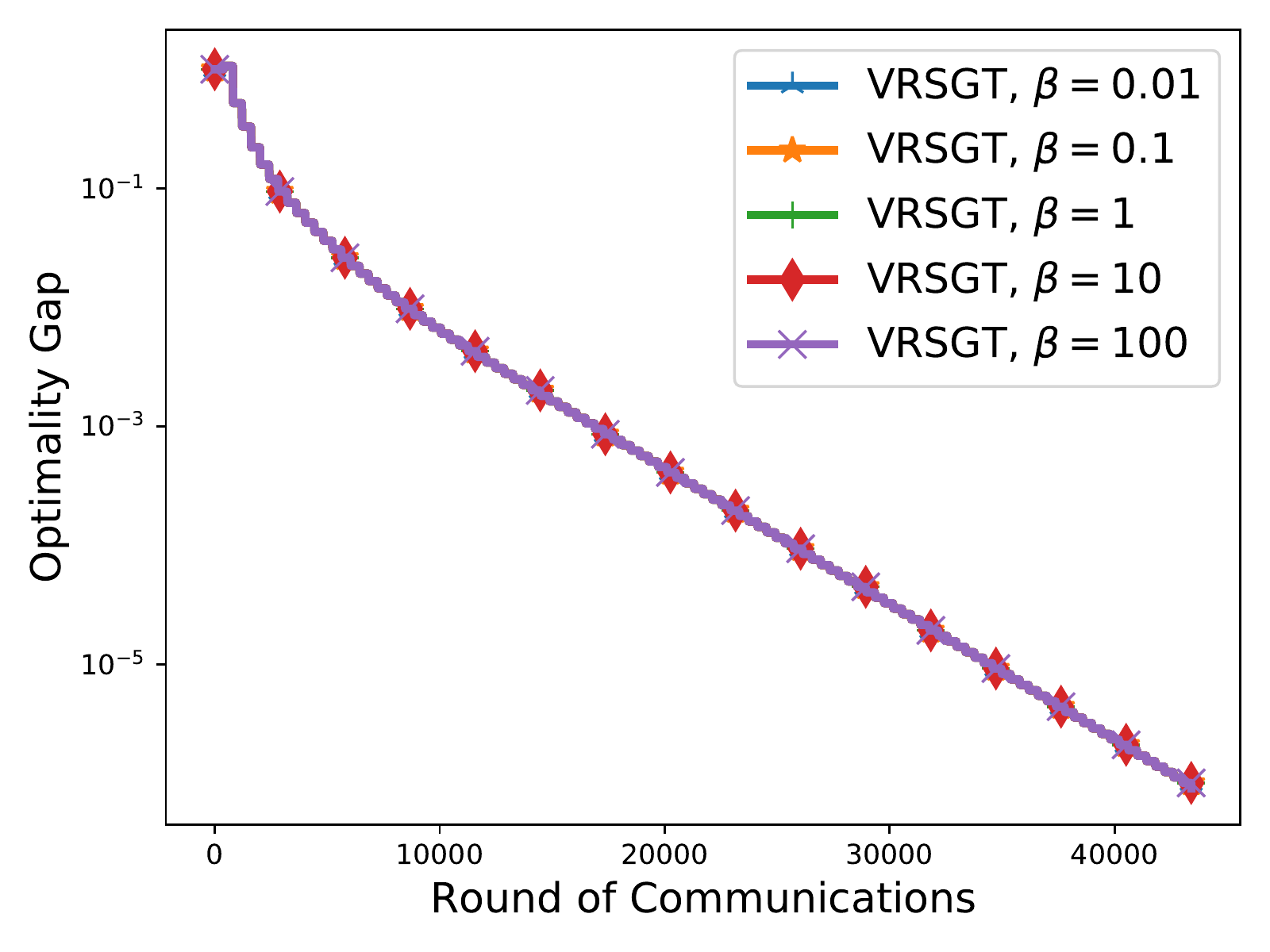}
	
	\caption{Comparison of VRSGT for different values of $\beta$.}
	\label{fig:PCA_beta}
\end{figure}

In the following experiments, the global data matrix $A\in\Rnm$ 
(assuming $n\leq m$ without loss of generality)
is constructed by its (economy-form) singular value decomposition as follows.
\begin{equation}\label{eq:gen-A}
	A = U \Sigma V\zz,
\end{equation}
where both $U \in \Rnn$ and $V \in \Rmn$ are orthogonal matrices 
orthonormalized from randomly generated matrices,
and $\Sigma \in \Rnn$ is a diagonal matrix with diagonal entries 
\begin{equation}\label{eq:Sigma_ii}
	\Sigma_{ii}=\xi^{i / 2}, \quad \iin,
\end{equation}  
for a parameter $\xi \in (0, 1)$ that determines the decay rate of the singular values of $A$.
In general, smaller decay rates (with $\xi$ closer to 1) correspond to more difficult cases. 
Finally, the global data matrix $A$ is uniformly distributed into $d$ agents.



Figure \ref{fig:PCA_beta} depicts the performances of VRSGT 
with different values of $\beta$,
which are distinguished by different colors.
A synthetic matrix $A$, 
generated by \eqref{eq:gen-A} with $n = 200$, $m = 64000$, and $\xi = 0.9$,
is tested with $p = 10$, $d = 32$, and $\eta = 0.0001$.
We can observe that the curves in Figure \ref{fig:PCA_beta} 
almost coincide with each other,
which indicates that VRSGT has almost the same performance
in a wide range of penalty parameters.
Therefore, in the following experiments, we fix $\beta = 1$ by default.

\subsection{Comparison on PCA Problems}

We first compare the performance of VRSGT 
with the existing algorithm DRSGD \cite{Chen2021decentralized} 
in solving the decentralized PCA problem
on the MNIST dataset \cite{LeCun1998gradient}.
The corresponding numerical experiments are performed on three different networks 
with $d = 16$ and $p = 5$,
including ring network, star network, 
and Erdos-Renyi network.
In the Erdos-Renyi network, 
two agents are connected with a fixed probability $\mathtt{prob} = 0.5$.
And all the networks are associated with the Metropolis constant matrix \cite{Shi2015}
as the mixing matrix $W$.
The batch size is set to $10$ for both VRSGT and DRSGD.
Figure \ref{fig:PCA_MNIST} depicts the decay of the stationarity gap defined in \eqref{eq:optgap}
versus the round of communications.
We can observe that VRSGT outperforms DRSGD in all the three networks.

\begin{figure}[ht!]
	\centering
	
	\subfigure[Ring network]{
		\label{subfig:PCA_kkt_beta}
		\includegraphics[width=0.3\linewidth]{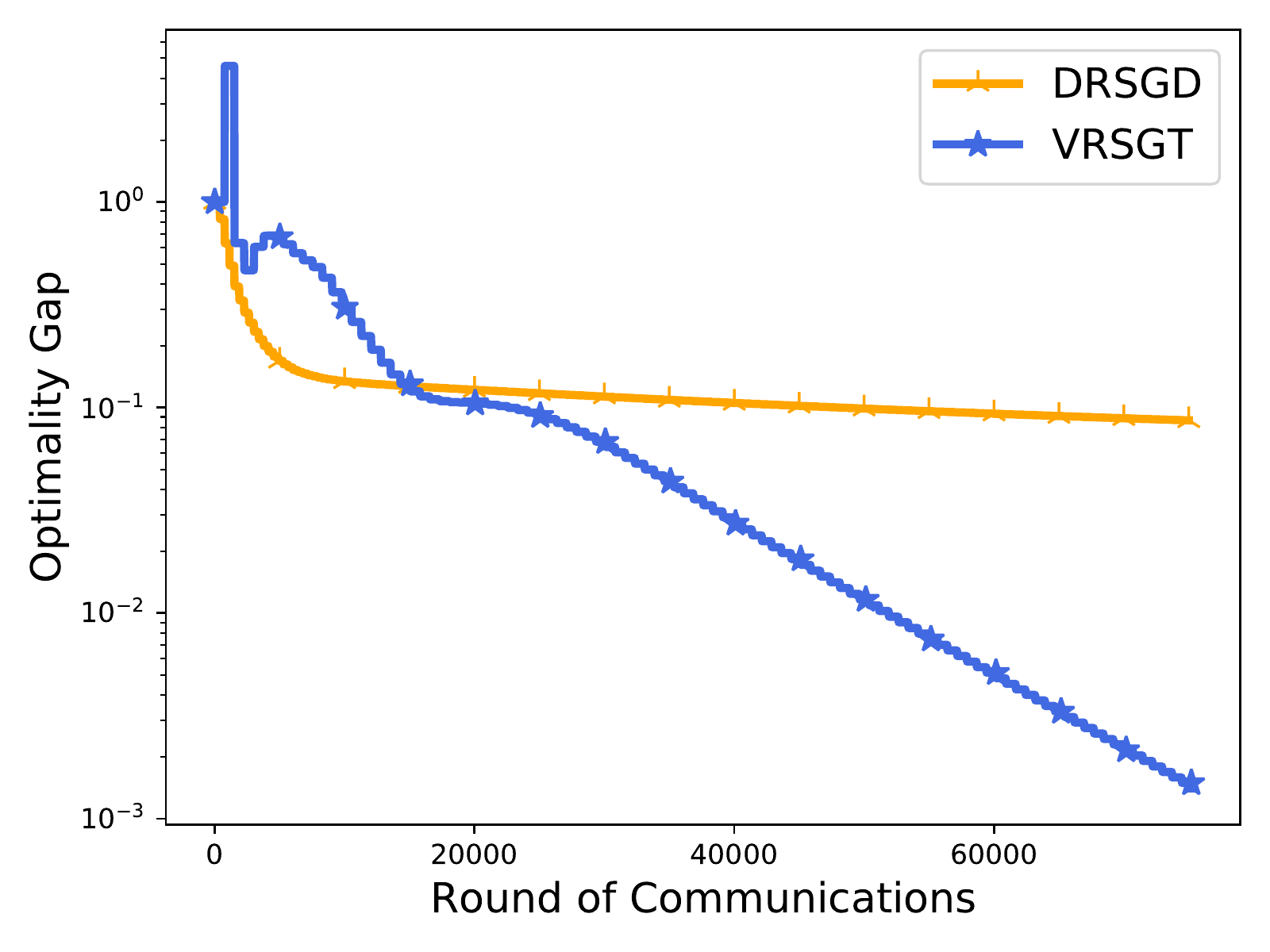}
	}
	\subfigure[Star network]{
		\label{subfig:PCA_cons_beta}
		\includegraphics[width=0.3\linewidth]{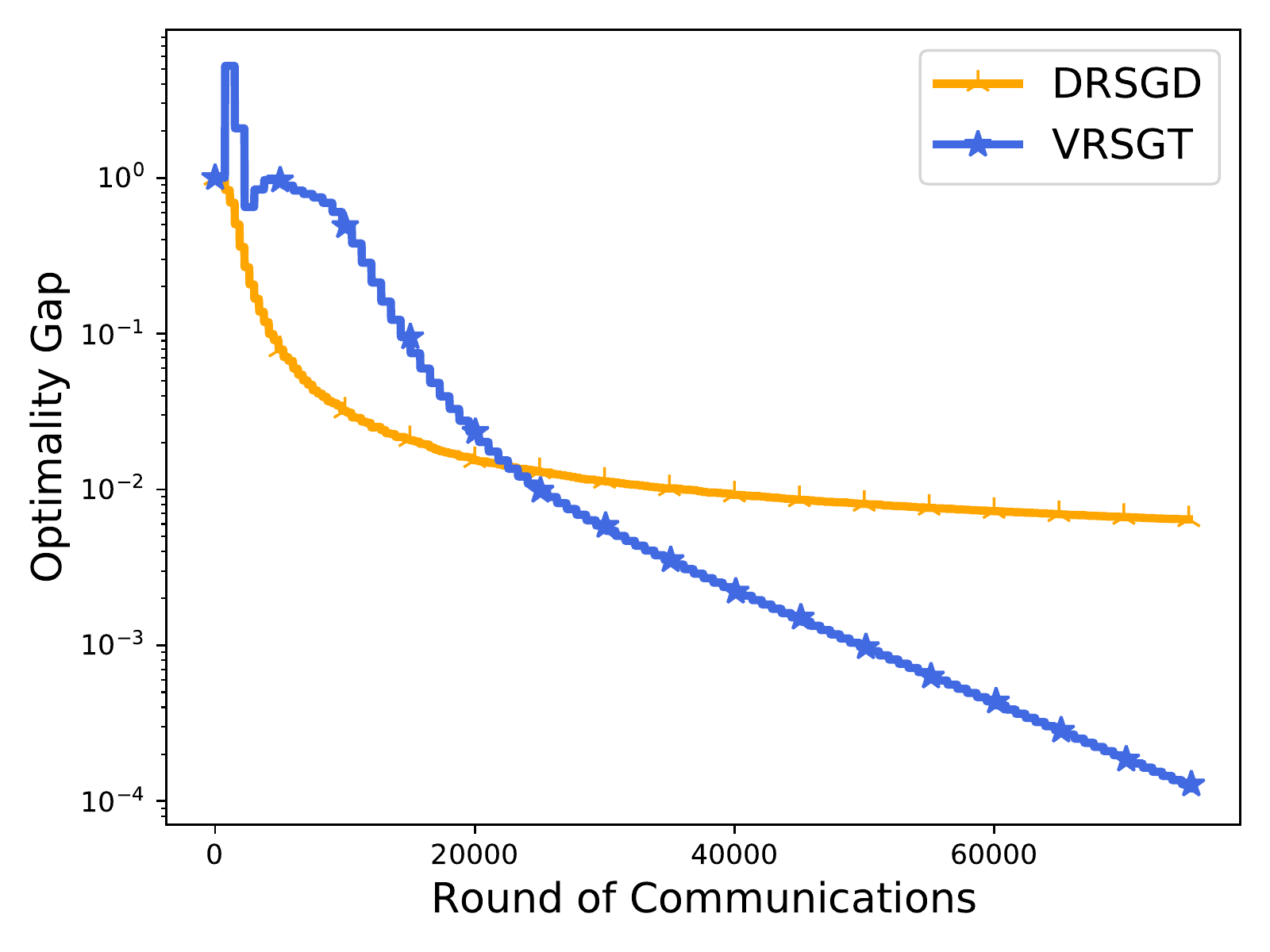}
	}
	\subfigure[Erdos-Renyi network]{
		\label{subfig:PCA_feas_beta}
		\includegraphics[width=0.3\linewidth]{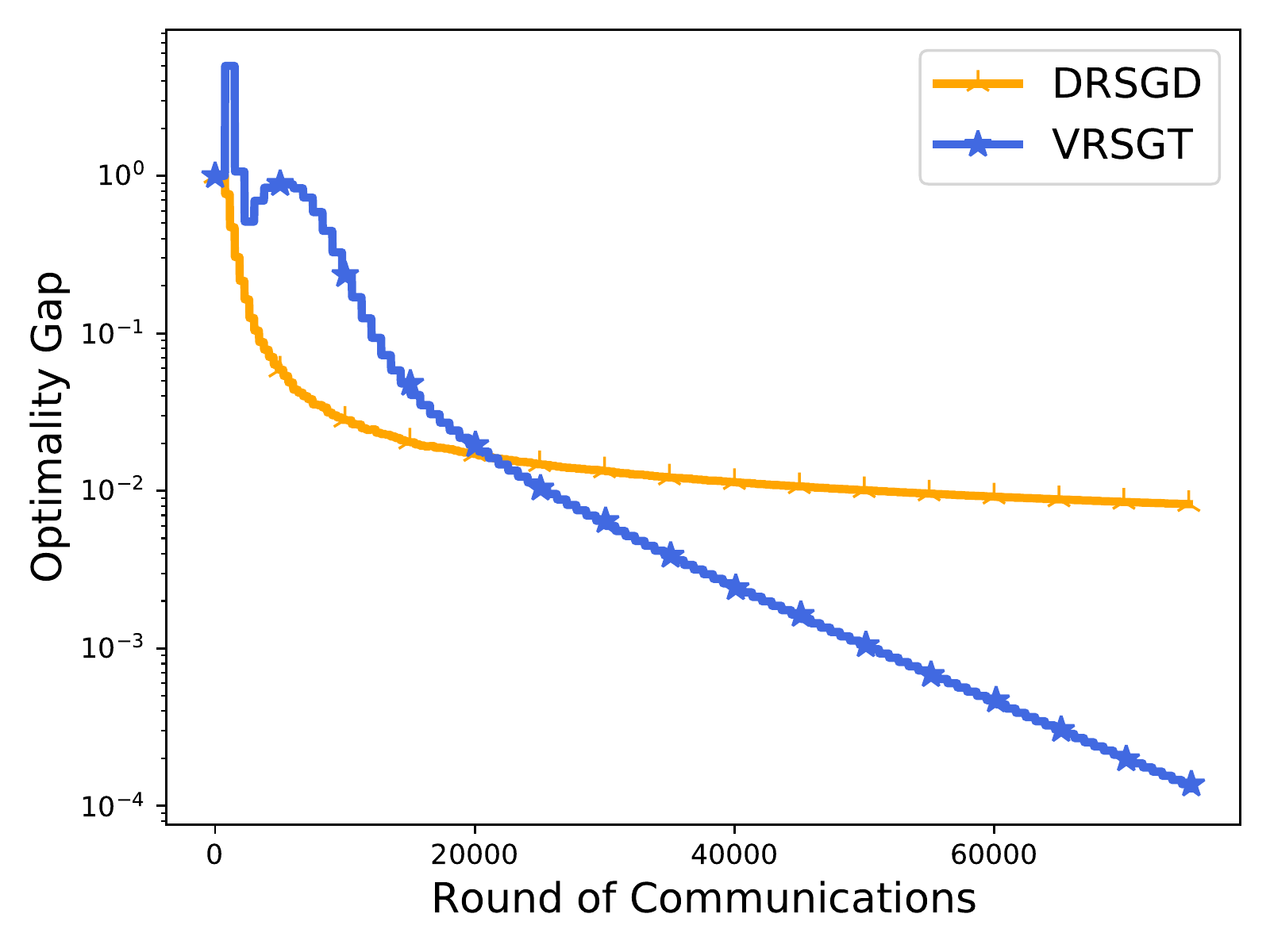}
	}
	
	\caption{Comparison between VRSGT and DRSGD in solving the decentralized PCA problem.}
	\label{fig:PCA_MNIST}
\end{figure}

\subsection{Comparison on DPCP Problems}

Our next experiment is to test VRSGT and DRSGD in solving the decentralized DPCP problem
on the KITTI dataset \cite{Geiger2013vision},
which plays an important role in autonomous driving applications.
DPCP is employed to distinguish the 3D points lying on the road plane (inliers)
and other objects off that plane (outliers).
Each point cloud in the KITTI dataset 
contains around $10^5$ points with approximately 50\% outliers.
And the data is homogenized and normalized to unit $\ell_2$-norm. 
Specifically, given a 3D point cloud of a road scene, 
DPCP reconstructs an affine plane $\{x \in \bR^3 \mid a\zz x - b = 0\}$ 
as a representation of the road. 
This task can be converted to a linear subspace learning problem 
by embedding the affine plane from $\bR^3$ into a hyperplane from $\bR^4$ 
with a normal vector $w = [a, -b]$ 
through the mapping $x \to [x\zz, 1]\zz$ \cite{Ding2019noisy}.

\begin{figure}[ht!]
	\centering
	
	\subfigure[Frame 1 of KITTI-CITY-5]{
		\includegraphics[width=0.45\linewidth]{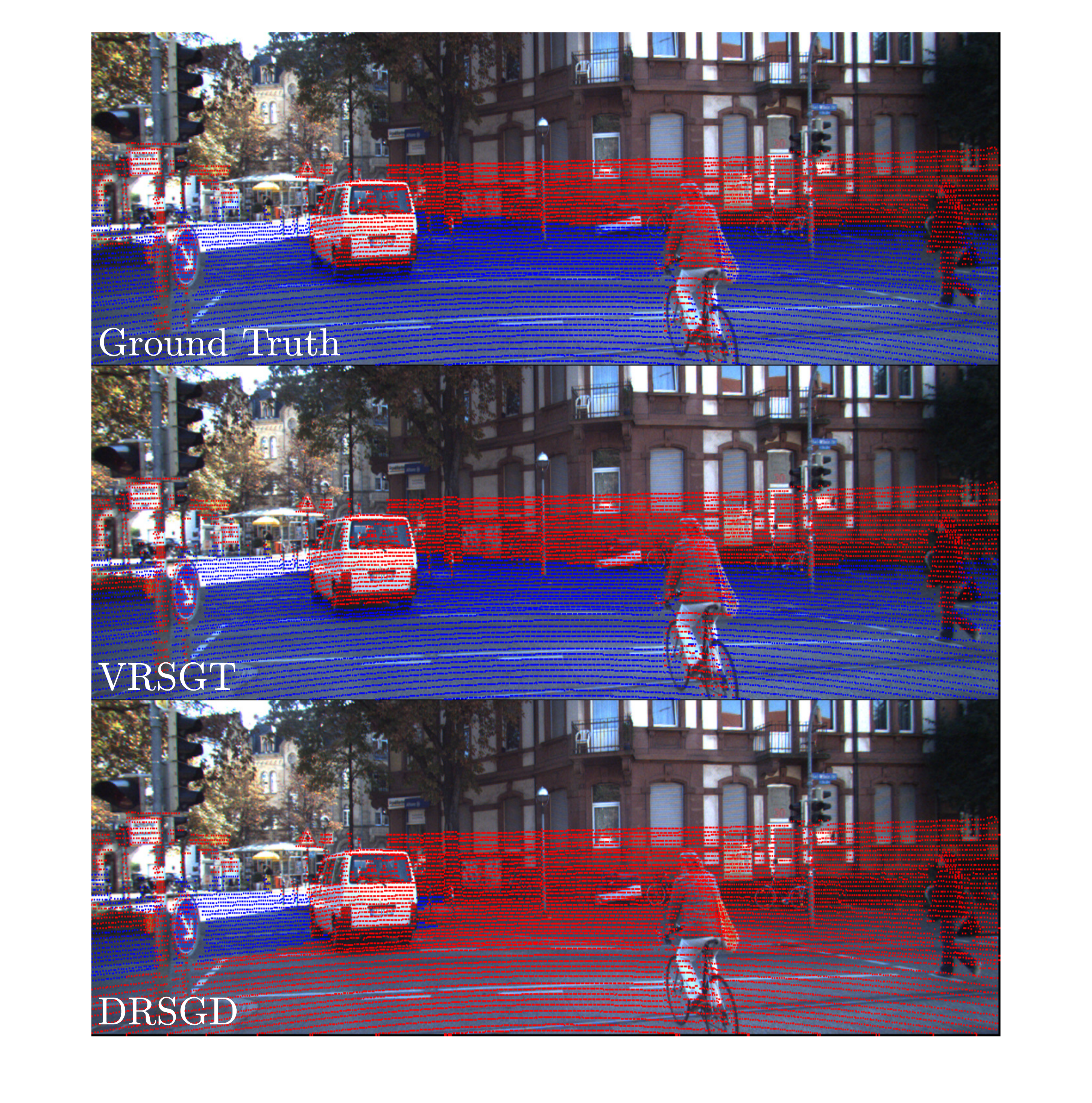}
	}
	\subfigure[Frame 45 of KITTI-CITY-5]{
		\includegraphics[width=0.45\linewidth]{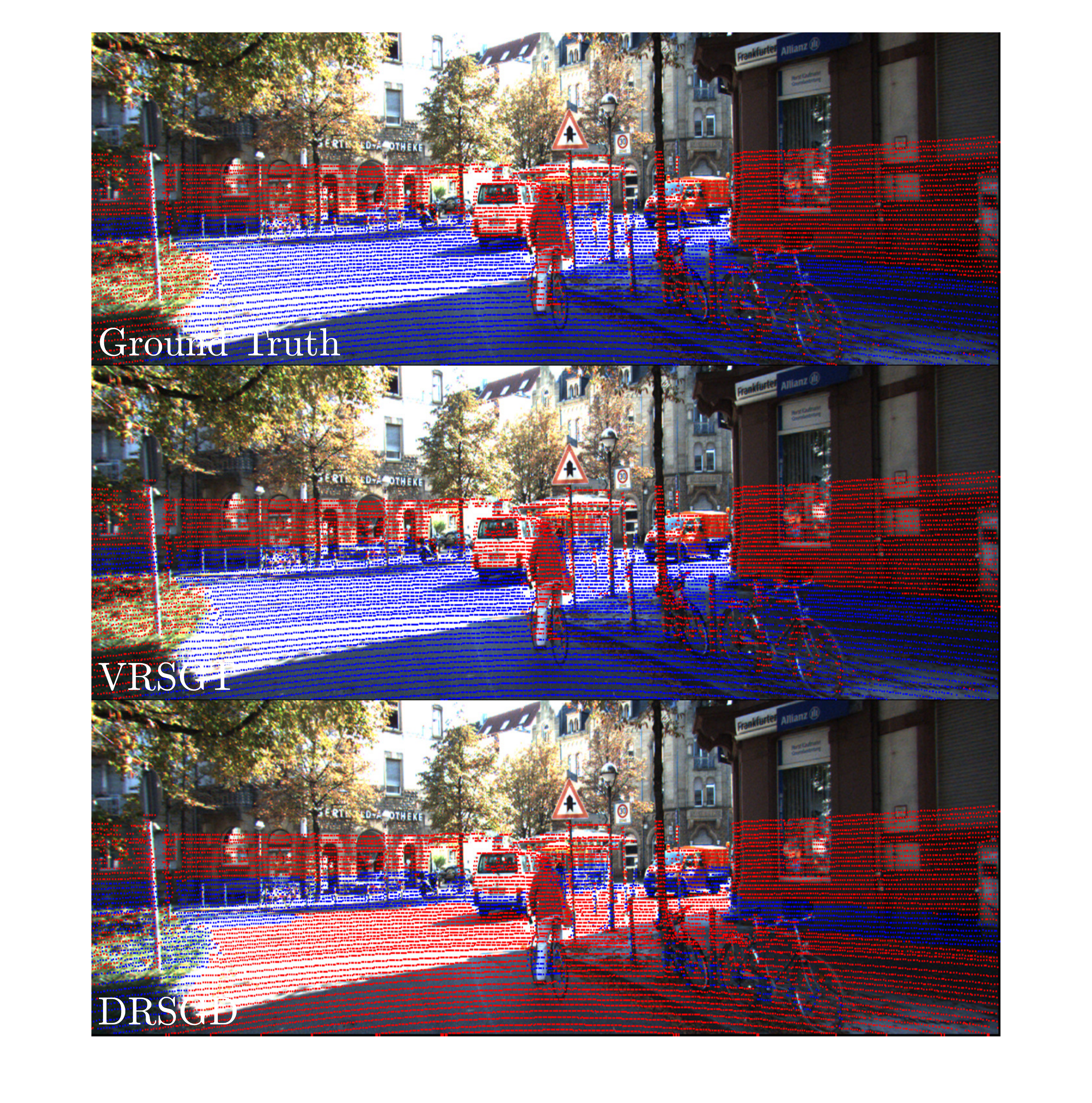}
	}
	
	\subfigure[Frame 0 of KITTI-CITY-48]{
		\includegraphics[width=0.45\linewidth]{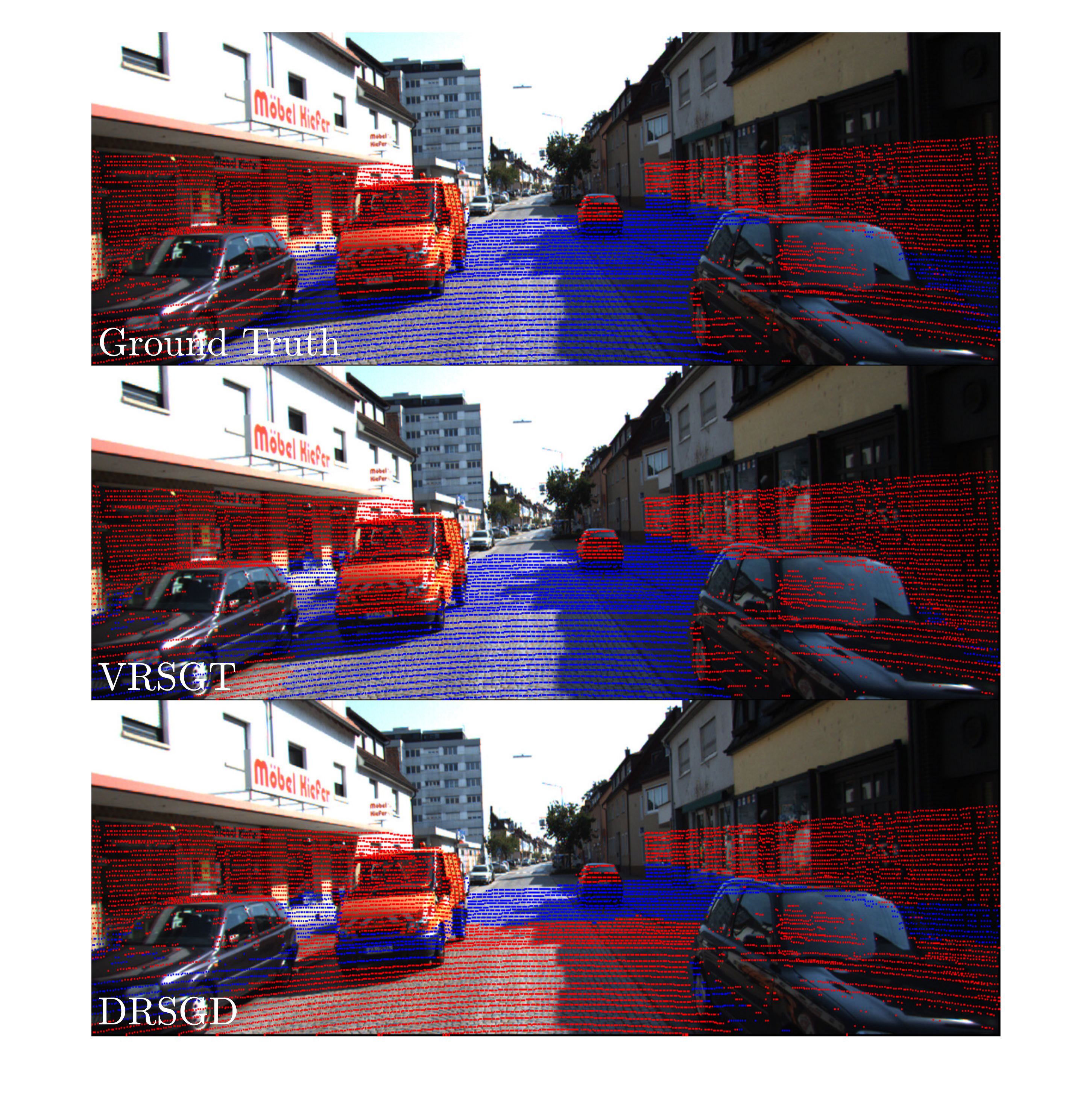}
	}
	\subfigure[Frame 328 of KITTI-CITY-71]{
		\includegraphics[width=0.45\linewidth]{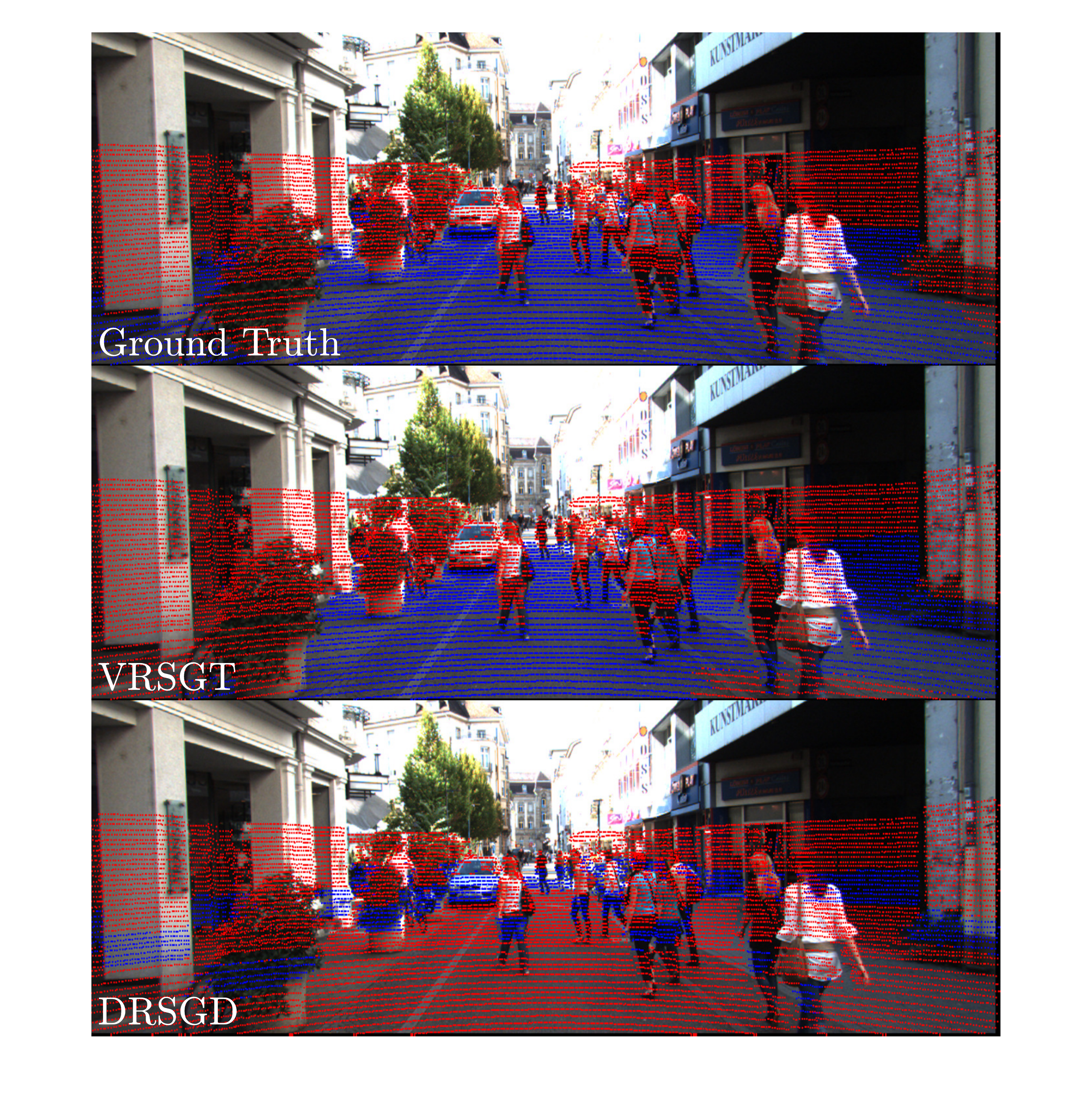}
	}
	
	\caption{Recovery results of four frames from the KITTI dataset 
		with inliers in blue and outliers in red. 
		Both inliers and outliers are detected by using a threshold on the distance 
		to the hyperplane recovered by each tested algorithm. 
		The results are represented by projecting 3D point clouds onto the image.}
	\label{fig:KITTI}
\end{figure}

In this test, VRSGT and DRSGD are set to run $100$ and $200$ epochs, respectively,
such that they perform the same rounds of communications.
And the batch size is set to $16$ for both VRSGT and DRSGD.
This experiment is conducted on an Erdos-Renyi network 
with $\mathtt{prob} = 0.5$ and $d = 16$.
The Metropolis constant matrix is also associated with the network
as the mixing matrix $W$.
We select four frames from the KITTI datset to test.
Table \ref{tb:KITTI} tabulates the corresponding recovery error
$\sqrt{1 - \langle \bar{w}, w^{\ast} \rangle^2}$ for each frame,
where $\bar{w}$ is the normal vector recovered by the tested algorithm
and $w^{\ast}$ denotes the ground truth.
Moreover, we present the recovery results of these four frames in Figure \ref{fig:KITTI}.
Both Table \ref{tb:KITTI} and Figure \ref{fig:KITTI} illustrate that
VRSGT produces better classification accuracy than DRSGD.

\begin{table}[ht!]
	\caption{Recovery errors of VRSGT and DRSGD on selected frames from the KITTI dataset.}
	\label{tb:KITTI}
	\centering
	\footnotesize
	\begin{tabular}{c|c|c|c|c} 
		\toprule 
		Dataset & \multicolumn{2}{c|}{KITTI-CITY-5} & KITTI-CITY-48 & KITTI-CITY-71 \\
		\midrule
		Frame & 1  & 45 & 0 & 328   \\
		\midrule
		VRSGT & 1.074e-02  & 7.997e-03 & 3.399e-02 & 4.088e-02 \\
		\midrule
		DRSGD & 1.063e-01  & 9.618e-02 & 1.646e-01 & 2.144e-01 \\
		\bottomrule 
	\end{tabular}
\end{table}

\section{Concluding Remarks}

\label{sec:concluding}

In this paper, we propose the algorithm VRSGT 
to minimize a finite sum of smooth functions with orthogonality constraints,
which are distributed over a decentralized network of multiple agents.
With appropriate algorithmic parameters, 
VRSGT achieves significantly improved sampling and communication complexities
compared with existing decentralized Riemannian methods \cite{Chen2021decentralized}.
In the future work, we will focus on how to accelerate the training process of
deep neural networks with orthogonality constraints
\cite{Arjovsky2016,Vorontsov2017,Huang2018}. 
Here, the main difficulty is to tackle the nonsmoothness and possible non-regularity of the objective function.


\bibliographystyle{siam}

\bibliography{library}

\begin{thebibliography}{10}

\bibitem{Absil2008}
{\sc P.-A. Absil, R.~Mahony, and R.~Sepulchre}, {\em {Optimization algorithms
  on matrix manifolds}}, Princeton University Press, 2008.

\bibitem{Arjovsky2016}
{\sc M.~Arjovsky, A.~Shah, and Y.~Bengio}, {\em {Unitary evolution recurrent
  neural networks}}, in International Conference on Machine Learning, PMLR,
  2016, pp.~1120--1128.

\bibitem{Chang2020distributed}
{\sc T.-H. Chang, M.~Hong, H.-T. Wai, X.~Zhang, and S.~Lu}, {\em {Distributed
  learning in the nonconvex world: From batch data to streaming and beyond}},
  IEEE Signal Processing Magazine, 37 (2020), pp.~26--38.

\bibitem{Chang2015multi}
{\sc T.-H. Chang, M.~Hong, and X.~Wang}, {\em {Multi-agent distributed
  optimization via inexact consensus ADMM}}, IEEE Transactions on Signal
  Processing, 63 (2015), pp.~482--497.

\bibitem{Chen2021decentralized}
{\sc S.~Chen, A.~Garcia, M.~Hong, and S.~Shahrampour}, {\em {Decentralized
  Riemannian gradient descent on the Stiefel manifold}}, in Proceedings of the
  38th International Conference on Machine Learning, vol.~139, PMLR, 2021,
  pp.~1594--1605.

\bibitem{Ding2019noisy}
{\sc T.~Ding, Z.~Zhu, T.~Ding, Y.~Yang, D.~P. Robinson, M.~C. Tsakiris, and
  R.~Vidal}, {\em Noisy dual principal component pursuit}, in Proceedings of
  the 36th International Conference on Machine Learning, 2019, pp.~1617--1625.

\bibitem{Gao2019}
{\sc B.~Gao, X.~Liu, and Y.-X. Yuan}, {\em {Parallelizable algorithms for
  optimization problems with orthogonality constraints}}, {SIAM} Journal on
  Scientific Computing, 41 (2019), pp.~A1949--A1983.

\bibitem{Geiger2013vision}
{\sc A.~Geiger, P.~Lenz, C.~Stiller, and R.~Urtasun}, {\em {Vision meets
  robotics: The KITTI dataset}}, The International Journal of Robotics
  Research, 32 (2013), pp.~1231--1237.

\bibitem{Hajinezhad2019}
{\sc D.~Hajinezhad and M.~Hong}, {\em {Perturbed proximal primal--dual
  algorithm for nonconvex nonsmooth optimization}}, Mathematical Programming,
  176 (2019), pp.~207--245.

\bibitem{Hartley2003multiple}
{\sc R.~Hartley and A.~Zisserman}, {\em {Multiple view geometry in computer
  vision}}, Cambridge University Press, 2003.

\bibitem{Hong2017}
{\sc M.~Hong, D.~Hajinezhad, and M.-M. Zhao}, {\em {Prox-PDA: The proximal
  primal-dual algorithm for fast distributed nonconvex optimization and
  learning over networks}}, in International Conference on Machine Learning,
  PMLR, 2017, pp.~1529--1538.

\bibitem{Hu2020brief}
{\sc J.~Hu, X.~Liu, Z.-W. Wen, and Y.-X. Yuan}, {\em {A brief introduction to
  manifold optimization}}, Journal of the Operations Research Society of China,
  8 (2020), pp.~199--248.

\bibitem{Hu2020efficient}
{\sc X.~Hu and X.~Liu}, {\em {An efficient orthonormalization-free approach for
  sparse dictionary learning and dual principal component pursuit}}, Sensors,
  20 (2020), p.~3041.

\bibitem{Huang2018}
{\sc L.~Huang, X.~Liu, B.~Lang, A.~W. Yu, Y.~Wang, and B.~Li}, {\em {Orthogonal
  weight normalization: Solution to optimization over multiple dependent
  Stiefel manifolds in deep neural networks}}, in Thirty-Second AAAI Conference
  on Artificial Intelligence, 2018.

\bibitem{Huang2020}
{\sc L.-K. Huang and S.~Pan}, {\em {Communication-efficient distributed PCA by
  Riemannian optimization}}, in International Conference on Machine Learning,
  PMLR, 2020, pp.~4465--4474.

\bibitem{Jiang2017vector}
{\sc B.~Jiang, S.~Ma, A.~M.-C. So, and S.~Zhang}, {\em {Vector transport-free
  SVRG with general retraction for Riemannian optimization: Complexity analysis
  and practical implementation}}, arXiv:1705.09059,  (2017).

\bibitem{Kempe2008}
{\sc D.~Kempe and F.~McSherry}, {\em {A decentralized algorithm for spectral
  analysis}}, Journal of Computer and System Sciences, 74 (2008), pp.~70--83.

\bibitem{LeCun1998gradient}
{\sc Y.~LeCun, L.~Bottou, Y.~Bengio, and P.~Haffner}, {\em {Gradient-based
  learning applied to document recognition}}, Proceedings of the IEEE, 86
  (1998), pp.~2278--2324.

\bibitem{Lian2017can}
{\sc X.~Lian, C.~Zhang, H.~Zhang, C.-J. Hsieh, W.~Zhang, and J.~Liu}, {\em {Can
  decentralized algorithms outperform centralized algorithms? A case study for
  decentralized parallel stochastic gradient descent}}, Advances in Neural
  Information Processing Systems, 30 (2017).

\bibitem{Ling2015}
{\sc Q.~Ling, W.~Shi, G.~Wu, and A.~Ribeiro}, {\em {DLM: Decentralized
  linearized alternating direction method of multipliers}}, IEEE Transactions
  on Signal Processing, 63 (2015), pp.~4051--4064.

\bibitem{Liu2019b}
{\sc H.~Liu, A.~M.-C. So, and W.~Wu}, {\em {Quadratic optimization with
  orthogonality constraint: explicit {\L}ojasiewicz exponent and linear
  convergence of retraction-based line-search and stochastic variance-reduced
  gradient methods}}, Mathematical Programming, 178 (2019), pp.~215--262.

\bibitem{Nedic2018network}
{\sc A.~Nedi{\'c}, A.~Olshevsky, and M.~G. Rabbat}, {\em {Network topology and
  communication-computation tradeoffs in decentralized optimization}},
  Proceedings of the IEEE, 106 (2018), pp.~953--976.

\bibitem{Nedic2009}
{\sc A.~Nedic and A.~Ozdaglar}, {\em {Distributed subgradient methods for
  multi-agent optimization}}, IEEE Transactions on Automatic Control, 54
  (2009), pp.~48--61.

\bibitem{Penna2014}
{\sc F.~Penna and S.~Sta{\'n}czak}, {\em {Decentralized eigenvalue algorithms
  for distributed signal detection in wireless networks}}, IEEE Transactions on
  Signal Processing, 63 (2014), pp.~427--440.

\bibitem{Pillai2005perron}
{\sc S.~U. Pillai, T.~Suel, and S.~Cha}, {\em {The Perron-Frobenius theorem:
  some of its applications}}, IEEE Signal Processing Magazine, 22 (2005),
  pp.~62--75.

\bibitem{Qi2010riemannian}
{\sc C.~Qi, K.~A. Gallivan, and P.-A. Absil}, {\em {Riemannian BFGS algorithm
  with applications}}, in Recent advances in optimization and its applications
  in engineering, Springer, 2010, pp.~183--192.

\bibitem{Qu2017}
{\sc G.~Qu and N.~Li}, {\em {Harnessing smoothness to accelerate distributed
  optimization}}, IEEE Transactions on Control of Network Systems, 5 (2017),
  pp.~1245--1260.

\bibitem{Raja2015}
{\sc H.~Raja and W.~U. Bajwa}, {\em {Cloud K-SVD: A collaborative dictionary
  learning algorithm for big, distributed data}}, IEEE Transactions on Signal
  Processing, 64 (2015), pp.~173--188.

\bibitem{Sato2019riemannian}
{\sc H.~Sato, H.~Kasai, and B.~Mishra}, {\em {Riemannian stochastic variance
  reduced gradient algorithm with retraction and vector transport}}, SIAM
  Journal on Optimization, 29 (2019), pp.~1444--1472.

\bibitem{Shi2015}
{\sc W.~Shi, Q.~Ling, G.~Wu, and W.~Yin}, {\em {EXTRA: An exact first-order
  algorithm for decentralized consensus optimization}}, SIAM Journal on
  Optimization, 25 (2015), pp.~944--966.

\bibitem{Stiefel1935}
{\sc E.~Stiefel}, {\em {Richtungsfelder und fernparallelismus in
  n-dimensionalen mannigfaltigkeiten}}, Commentarii Mathematici Helvetici, 8
  (1935), pp.~305--353.

\bibitem{Sun2020improving}
{\sc H.~Sun, S.~Lu, and M.~Hong}, {\em {Improving the sample and communication
  complexity for decentralized non-convex optimization: Joint gradient
  estimation and tracking}}, in Proceedings of the 37th International
  conference on machine learning, PMLR, 2020, pp.~9217--9228.

\bibitem{Vorontsov2017}
{\sc E.~Vorontsov, C.~Trabelsi, S.~Kadoury, and C.~Pal}, {\em {On orthogonality
  and learning recurrent networks with long term dependencies}}, in
  International Conference on Machine Learning, PMLR, 2017, pp.~3570--3578.

\bibitem{Wang2021multipliers}
{\sc L.~Wang, B.~Gao, and X.~Liu}, {\em {Multipliers correction methods for
  optimization problems over the Stiefel manifold}}, CSIAM Transactions on
  Applied Mathematics, 2 (2021), pp.~508--531.

\bibitem{Xiao2004}
{\sc L.~Xiao and S.~Boyd}, {\em {Fast linear iterations for distributed
  averaging}}, Systems \& Control Letters, 53 (2004), pp.~65--78.

\bibitem{Xiao2020}
{\sc N.~Xiao, X.~Liu, and Y.-X. Yuan}, {\em {A class of smooth exact penalty
  function methods for optimization problems with orthogonality constraints}},
  Optimization Methods and Software,  (2020), pp.~1--37.

\bibitem{Xin2022fast}
{\sc R.~Xin, U.~A. Khan, and S.~Kar}, {\em {Fast decentralized nonconvex
  finite-sum optimization with recursive variance reduction}}, SIAM Journal on
  Optimization, 32 (2022), pp.~1--28.

\bibitem{Xin2020general}
{\sc R.~Xin, S.~Pu, A.~Nedi{\'c}, and U.~A. Khan}, {\em {A general framework
  for decentralized optimization with first-order methods}}, Proceedings of the
  IEEE, 108 (2020), pp.~1869--1889.

\bibitem{Xu2015}
{\sc J.~Xu, S.~Zhu, Y.~C. Soh, and L.~Xie}, {\em {Augmented distributed
  gradient methods for multi-agent optimization under uncoordinated constant
  stepsizes}}, in 54th IEEE Conference on Decision and Control, IEEE, 2015,
  pp.~2055--2060.

\bibitem{Yuan2016}
{\sc K.~Yuan, Q.~Ling, and W.~Yin}, {\em {On the convergence of decentralized
  gradient descent}}, SIAM Journal on Optimization, 26 (2016), pp.~1835--1854.

\bibitem{Zeng2018}
{\sc J.~Zeng and W.~Yin}, {\em On nonconvex decentralized gradient descent},
  IEEE Transactions on Signal Processing, 66 (2018), pp.~2834--2848.

\bibitem{Zhang2016riemannian}
{\sc H.~Zhang, S.~J~Reddi, and S.~Sra}, {\em {Riemannian SVRG: Fast stochastic
  optimization on Riemannian manifolds}}, Advances in Neural Information
  Processing Systems, 29 (2016).

\bibitem{Zhu2018dual}
{\sc Z.~Zhu, Y.~Wang, D.~Robinson, D.~Naiman, R.~Vidal, and M.~Tsakiris}, {\em
  {Dual principal component pursuit: Improved analysis and efficient
  algorithms}}, Advances in Neural Information Processing Systems, 31 (2018).

\end{thebibliography}

\addcontentsline{toc}{section}{References}



\begin{appendices}
	
%
%

\section{Proof of Lemma \ref{le:expen}.}

\label{apx:expen}

\begin{proof}[{\bf Proof of Lemma \ref{le:expen}}]
	Suppose $\sigma_{\min}$ is the smallest singular value of $X$.
	Then for any $X \in \cR$, we have $\norm{X}_2^2 \leq 7 / 6$ 
	and $\sigma_{\min}^2 \geq 5 / 6$, and hence,
	\begin{equation*}
		\norm{X (X\zz X - I_p)}\fs 
		\geq \sigma_{\min}^2 \norm{X\zz X - I_p}\fs
		\geq \dfrac{5}{6} \norm{X\zz X - I_p}\fs.
	\end{equation*}
	Moreover, simple algebraic manipulations give us
	\begin{equation*}
		\begin{aligned}
			& \jkh{G (X), X (X\zz X - I_p)} \\
			= {} & \jkh{\nabla f (Z) |_{Z = X} 
				\dkh{\dfrac{3}{2}I_p - \dfrac{1}{2} X\zz X}, 
				X (X\zz X - I_p)} \\
			& - \jkh{X \sym \dkh{X\zz \nabla f (Z) |_{Z = X}}, 
				X (X\zz X - I_p)} \\
			= {} & \jkh{\sym \dkh{X\zz \nabla f_i (Z) |_{Z = X}}, 
				(X\zz X - I_p) \dkh{\dfrac{3}{2}I_p - \dfrac{1}{2} X\zz X} 
				- X\zz X (X\zz X - I_p)} \\
			= {} & - \dfrac{3}{2} 
			\jkh{\sym \dkh{X\zz \nabla f (Z) |_{Z = X}}, 
				(X\zz X - I_p)^2},
		\end{aligned}
	\end{equation*}
	which yields that
	\begin{equation*}
		\begin{aligned}
			\abs{\jkh{G (X), X (X\zz X - I_p)}}
			& \leq \dfrac{3}{2} 
			\norm{\sym \dkh{X\zz \nabla f (Z) |_{Z = X}}}\ff 
			\norm{(X\zz X - I_p)^2}\ff \\
			& \leq \dfrac{7}{4} M \norm{X\zz X - I_p}\fs.
		\end{aligned}
	\end{equation*}
	Now it can be readily verifies that
	\begin{equation*}
		\begin{aligned}
			\norm{H (X)}\fs 
			& = \norm{G (X)}\fs 
			+ 2 \beta \jkh{G (X), X (X\zz X - I_p)}
			+ \beta^2 \norm{X (X\zz X - I_p)}\fs \\
			& \geq \norm{G (X)}\fs 
			+ \dfrac{1}{6} \beta \dkh{5 \beta - 21 M}\norm{X\zz X - I_p}\fs,
		\end{aligned}
	\end{equation*}
	which together with $\beta \geq (6 + 21 M) / 5$ completes the proof.
\end{proof}

\section{Convergence Analysis}

\label{apx:convergence}

Existing convergence guarantees of decentralized SVRG algorithms, 
such as \cite{Sun2020improving,Xin2022fast}, 
are constructed for globally Lipschitz smooth functions,
which are restrictive for \eqref{eq:opt-dest} since $\Snp$ is compact.
In this section, the global convergence of Algorithm~\ref{alg:VRSGT} 
is rigorously established under rather mild conditions.
The objective function is only assumed to be locally Lipschitz smooth 
(see Assumption \ref{asp:smooth}).

To begin with, according to the Perron-Frobenius Theorem \cite{Pillai2005perron},
the second largest eigenvalue in magnitude of the mixing matrix $W$, 
denoted by $\lambda$, is strictly less than 1.
Recall that $\cR$ is a bounded set defined in Lemma \ref{le:expen}.
And we define $\cB := \{X \in \Rnp \mid \norm{X}\ff \leq \sqrt{7dp / 6} + \sqrt{d}\}$.

All the special constants to be used in this section are listed below.
We divide these constants into four categories.

\begin{itemize}
	
	\item Category I:
	\begin{equation} \label{eq:constants-1}
		\begin{aligned}
			& M_G = \sup\, \{ \norm{G^{[j]}_i (X)}\ff \mid X \in \cB, \iid, j \in [l] \}; \quad
			M_E = \sup\, \{ \norm{bE(X)}\ff \mid X \in \cB\}; \\
			& L_G = \sup\, \hkh{ 
				\dfrac{\norm{G^{[j]}_i (X) - G^{[j]}_i (Y)}\ff}{\norm{X - Y}\ff}
				\bigg| X \neq Y, X \in \cB, Y \in \cB, \iid, j \in [l]}; \\
			& L_E = \sup\, \hkh{ 
				\dfrac{\norm{E(X) - E (Y)}\ff}{\norm{X - Y}\ff}
				\bigg| X \neq Y, X \in \cB, Y \in \cB}; \\
			& M_D = \norm{\avD^{(0, 0)} - \bfD^{(0, 0)}}\ff; \quad
			\gamma = \dfrac{1 - \lambda^2}{2 \lambda^2}; \quad
			\underline{f} = \inf \{f (X) \mid X \in \cB\}.			
		\end{aligned}
	\end{equation}
	
	\item Category II:
	\begin{equation} \label{eq:constants-2}
		\begin{aligned}
			& \hat{L} = L_G + \beta L_E; \quad
			C_D = M_D + \sqrt{d} \dkh{(2q + 1) M_G + \beta M_E}; \\
			& C_1 = \dfrac{\eta}{2} - \dfrac{\eta^2 \hat{L}}{2} - 4 \eta^3 \hat{L}^2 - 24 \dkh{1 + \dfrac{1}{\gamma}} \eta^3 \hat{L}^2; \\
			& C_2 = \dfrac{1 - \lambda^2}{2} - 48 \dkh{1 + \dfrac{1}{\gamma}} \eta \hat{L}^2 - 9 \eta \hat{L}^2, \\
			& C_3 = \dfrac{\eta (1 - \lambda^2)}{2} - \dkh{1 + \dfrac{1}{\gamma}}\eta^2 - 24 \dkh{1 + \dfrac{1}{\gamma}} \eta^3 \hat{L}^2 - 4 \eta^3 \hat{L}^2; \\
			& C_4 = \dkh{f (X_{\mathrm{init}}) - \underline{f}}
			\dkh{\dfrac{8 \eta^2 \hat{L}^2 + 2}{C_1}
			+ \dfrac{16 \hat{L}^2 + 1}{d C_2}
			+ \dfrac{8 \eta^2 \hat{L}^2}{d C_3}}; \\
			& C = 2 (\hat{L}^2 + 1) C_4.
		\end{aligned}
	\end{equation}

	\item Category III:
	\begin{equation} \label{eq:constants-3}
		\begin{aligned}
			& \bar{\eta}_1 = \dfrac{216}{49 \beta^2}; \quad 
			\bar{\eta}_2 = \dfrac{12}{7 \beta}; \quad
			\bar{\eta}_3 = \dfrac{1}{18(4 + 3(2q + 1) M_G)}; \\
			& \bar{\eta}_4 = \dfrac{5 \beta / 72 - 4 - 3 (2q + 1) M_G}{(1 + (2q + 1) M_G)^2}; \quad
			\bar{\eta}_5 
			= \dfrac{\sqrt{d} (1 - \lambda)}{\beta^2 L_E \lambda C_D}; \quad
			\bar{\eta}_6 = \dfrac{\sqrt{d}}{\beta^2 L_E C_D}; \\
			& \bar{\eta}_7 = \dfrac{- \hat{L} / 2 + 
			\sqrt{\hat{L}^2 / 4 + 48 (1 + 1 / \gamma) \hat{L}^2 + 8 \hat{L}^2}}{48 (1 + 1 / \gamma) \hat{L}^2 + 8 \hat{L}^2}; \quad
			\bar{\eta}_8 = \dfrac{1 - \lambda^2}{96 (1 + 1 / \gamma) \hat{L}^2 + 18 \hat{L}^2}; \\
			& \bar{\eta}_9 = \dfrac{- (1 + 1 / \gamma) + \sqrt{(1 + 1 / \gamma)^2 + 8 (1 - \lambda^2) (6 (1 + 1 / \gamma) + \hat{L}^2)}}{48 (1 + 1 / \gamma) \hat{L}^2 + 8 \hat{L}^2}; \\
			& \bar{\eta} = \min \{\bar{\eta}_1, \bar{\eta}_2, \bar{\eta}_3, \bar{\eta}_4, \bar{\eta}_5, \bar{\eta}_6, \bar{\eta}_7, \bar{\eta}_8, \bar{\eta}_9\}.
		\end{aligned}
	\end{equation}

	\item Category IV:
	\begin{equation} \label{eq:constants-4}
		\begin{aligned}
			& \underline{\beta}_1 = \dfrac{6 + 21 M}{5}; \quad
			\underline{\beta}_2 = \dfrac{72 (4 + 3 (2q + 1) M_g)}{5}; \\
			& \underline{\beta}_3 = \sqrt{\dfrac{1}{L_E}}; \quad
			\underline{\beta}_4 = \dfrac{L_G}{L_E}; \quad
			\underline{\beta}_5 = \dfrac{6 \sqrt{d}}{(1 - \lambda) M_D}; \\
			& \underline{\beta} = \max \{1, \underline{\beta}_1, \underline{\beta}_2, \underline{\beta}_3, \underline{\beta}_4, \underline{\beta}_5\}.
		\end{aligned}
	\end{equation}
	
\end{itemize}

Category I consists of the constants defined in \eqref{eq:constants-1}
that is independent of $\beta$,
while the constants in Category II \eqref{eq:constants-2} depend on $\beta$.
Categories III and IV are composed of the constants related to the stepsize $\eta$ and penalty parameter $\beta$, respectively.  

In the following, we define an auxiliary sequence $\{T^{(k, t)}_i\}$ for VRSGT.
At the $k$-th outer iteration, $T^{(k + 1, 1)}_i$ is first updated by
\begin{equation*}
	T^{(k + 1, 1)}_i = G_{i} (X^{(k + 1, 1)}_i).
\end{equation*}
Then at the $t$-th inner iteration, $T^{(k + 1, t + 1)}_i$ is updated by
\begin{equation*}
	T^{(k + 1, t + 1)}_i = \dfrac{1}{\tau} \sum_{j \in \cC^{(k + 1, t)}} 
	\dkh{G^{[j]}_{i} (X^{(k + 1, t + 1)}_i) - G^{[j]}_{i} (X^{(k + 1, t)}_i)} 
	+ T^{(k + 1, t)}_i.
\end{equation*}
Now we prove that $S^{(k, t)}_i = T^{(k, t)}_i + \beta E (X^{(k, t)}_i)$ 
in the following lemma.

\begin{lemma} \label{le:T}
	Suppose $\{\bfX^{(k, t)}\}$ is the iterate sequence generated by Algorithm \ref{alg:VRSGT}.
	Then for any $k \in \bN$ and $t \in [q + 1]$,
	it holds that $S^{(k, t)}_i = T^{(k, t)}_i + \beta E (X^{(k, t)}_i)$.
\end{lemma}

\begin{proof}
	We use mathematical induction on $t$ to prove this lemma.
	At first, for $t = 1$, we have
	\begin{equation*}
		S^{(k, 1)}_i = H_{i} (X^{(k, 1)}_i) 
		= G_{i} (X^{(k, 1)}_i) + \beta E (X^{(k, 1)}_i)
		= T^{(k, 1)}_i + \beta E (X^{(k, 1)}_i).
	\end{equation*}
	Now, we assume that the assertion of this lemma holds for $S^{(k, t)}_i$ and $T^{(k, t)}_i$,
	and investigate the situation at for $S^{(k, t + 1)}_i$ and $T^{(k, t + 1)}_i$.
	In fact, it can be readily verified  that
	\begin{equation*}
		\begin{aligned}
			S^{(k, t + 1)}_i 
			= {} & \dfrac{1}{\tau} \sum_{j \in \cC^{(k, t)}} 
			\dkh{H^{[j]}_{i} (X^{(k, t + 1)}_i) - H^{[j]}_{i} (X^{(k , t)}_i)} 
			+ S^{(k, t)}_i \\
			= {} & \dfrac{1}{\tau} \sum_{j \in \cC^{(k, t)}} 
			\dkh{ G^{[j]}_{i} (X^{(k, t + 1)}_i) + \beta E (X^{(k, t + 1)}_i) 
				- G^{[j]}_{i} (X^{(k , t)}_i) - \beta E (X^{(k, t)}_i) } \\
			& + T^{(k, t)}_i + \beta E (X^{(k, t)}_i) \\
			= {} &  \dfrac{1}{\tau} \sum_{j \in \cC^{(k, t)}} 
			\dkh{ G^{[j]}_{i} (X^{(k, t + 1)}_i) - G^{[j]}_{i} (X^{(k , t)}_i) }
			+ T^{(k, t)}_i + \beta E (X^{(k, t + 1)}_i) \\
			= {} & T^{(k, t + 1)}_i + \beta E (X^{(k, t + 1)}_i).
		\end{aligned}
	\end{equation*}
	The proof is completed.
\end{proof}

Next, we prove that $T^{(k, t)}_i$ is bounded if $X^{(k, s)}_i \in \cB$ for $s \in [t]$.

\begin{lemma} \label{le:bound-T}
	Suppose $X^{(k, s)}_i \in \cB$ for $s \in [t]$.
	Then it holds that $\|T^{(k, t)}_i\|\ff \leq (2q + 1) M_g$.
\end{lemma}

\begin{proof}
	By straightforward manipulations, we have
	\begin{equation*}
		\begin{aligned}
			T^{(k, t)}_i 
			& = \sum_{s = 2}^t \dkh{ T^{(k, s)}_i - T^{(k, s - 1)}_i } + T^{(k, 1)}_i \\
			& = \dfrac{1}{\tau} \sum_{s = 2}^t \sum_{j \in \cC^{(k, s)}} 
			\dkh{G^{[j]}_{i} (X^{(k, s)}_i) - G^{[j]}_{i} (X^{(k, s - 1)}_i)} 
			+ G_{i} (X^{(k, 1)}_i),
		\end{aligned}
	\end{equation*}
	which is followed by
	\begin{equation*}
		\begin{aligned}
			\norm{ T^{(k, t)}_i }\ff
			& \leq \dfrac{1}{\tau} \sum_{s = 2}^t \sum_{j \in \cC^{(k, s)}} 
			\norm{G^{[j]}_{i} (X^{(k, s)}_i) - G^{[j]}_{i} (X^{(k, s - 1)}_i)} \ff
			+ \norm{G_{i} (X^{(k, 1)}_i)}\ff \\
			& \leq 2(t - 1) M_g + M_g
			\leq (2q + 1) M_g .
		\end{aligned}
	\end{equation*}
	We complete the proof.
\end{proof}

For the sake of convenience, we now define the following averages of local variables.
\begin{itemize}
	
	\item $\barX^{(k, t)} = \dfrac{1}{d} \sumiid X^{(k, t)}_i \in \Rnp$,
	$\avX^{(k, t)} = \dkh{\bfone_d \otimes I_n} \barX^{(k, t)} \in \bR^{dn \times p}$.

	\item $\barD^{(k, t)} =  \dfrac{1}{d} \sumiid D^{(k, t)}_i \in \Rnp$,
	$\avD^{(k, t)} = \dkh{\bfone_d \otimes I_n} \barD^{(k, t)} \in \bR^{dn \times p}$.
		
	\item $\barS^{(k, t)} =  \dfrac{1}{d} \sumiid S^{(k, t)}_i \in \Rnp$,
	$\avS^{(k, t)} = \dkh{\bfone_d \otimes I_n} \barS^{(k, t)} \in \bR^{dn \times p}$.
	
	\item $\barT^{(k, t)} =  \dfrac{1}{d} \sumiid T^{(k, t)}_i \in \Rnp$,
	$\avT^{(k, t)} = \dkh{\bfone_d \otimes I_n} \barT^{(k, t)} \in \bR^{dn \times p}$.
		
	\item $\barE^{(k, t)} =  \dfrac{1}{d} \sumiid E (X^{(k, t)}_i) \in \Rnp$,
	$\avE^{(k, t)} = \dkh{\bfone_d \otimes I_n} \barE^{(k, t)} \in \bR^{dn \times p}$.
		
\end{itemize}
Then it is not difficult to check that 
the relationship $\barD^{(k, t)} = \barS^{(k, t)} = \barT^{(k, t)} + \beta \barE^{(k, t)}$ 
holds for any $k \in \bN$ and $t \in [q + 1]$.
Moreover, we denote $J = \bfone_d \bfone_d\zz / d \in \bR^{d \times d}$,
and $\bfJ = J \otimes I_n \in \bR^{dn \times dn}$.
It holds that $\dkh{\bfW - \bfJ} \avX^{(k, t)} = \dkh{\bfW - \bfJ} \avD^{(k, t)} = 0$.

Based on Lemmas \ref{le:T} and \ref{le:bound-T}, 
we can prove the following technical lemma.

\begin{lemma} \label{le:bound-X-inner}
	Let the conditions in Assumptions \ref{asp:network} and \ref{asp:smooth} hold
	and the stepsize $\eta$ and penalty parameter $\beta$ of Algorithm \ref{alg:VRSGT} satisfy
	\begin{equation*}
		0 < \eta < \min\hkh{
			\bar{\eta}_1, \bar{\eta}_2, \bar{\eta}_3, \bar{\eta}_4
		},
		\mbox{~and~}
		\beta > \max \hkh{1, \underline{\beta}_2},
	\end{equation*}
	respectively.
	Suppose $\barX^{(k, t)} \in \cR$, 
	$\|T^{(k, t)}_i\|\ff \leq (2q + 1) M_G$,
	$\norm{\bfX^{(k, t)}}\ff \leq \sqrt{7dp / 6} + \sqrt{d}$,
	and $ \norm{\bar{\bfX}^{(k, t)} - \bfX^{(k, t)}}\ff \leq \sqrt{d} / (\beta^2 L_E)$.
	Then we have $\barX^{(k, t + 1)} \in \cR$.
\end{lemma}

\begin{proof}
	Since $\norm{\bfX^{(k, t)}}\ff \leq \sqrt{7dp / 6} + \sqrt{d}$,
	we have $X_i^{(k, t)} \in \cB$.
	Then it can be readily verified that
	\begin{equation*}
		\begin{aligned}
			\norm{ E (\barX^{(k, t)}) - \barE^{(k, t)} }\ff
			& \leq \dfrac{1}{d} \sumiid \norm{ E (\barX^{(k, t)}) 
				- E (X^{(k, t)}_i) }\ff \\
			& \leq \dfrac{L_E}{d} \sumiid \norm{\barX^{(k, t)} - X_i^{(k, t)}}\ff \\
			& \leq \dfrac{L_E}{\sqrt{d}} \norm{\bar{\bfX}^{(k, t)} - \bfX^{(k, t)}}\ff
			\leq \dfrac{1}{\beta}.
		\end{aligned}
	\end{equation*}
	Moreover, according to Lemma \ref{le:T}, we have
	\begin{equation*}
		\begin{aligned}
			\barX^{(k, t + 1)} 
			& = \barX^{(k, t)} - \eta \barD^{(k, t)} 
			=  \barX^{(k, t)} - \eta \barS^{(k, t)}
			= \barX^{(k, t)} - \eta (\barT^{(k, t)} + \beta \barE^{(k, t)}) \\
			& = \barX^{(k, t)} - \eta \beta E (\barX^{(k, t)}) 
			+ \eta \beta (E (\barX^{(k, t)}) - \barE^{(k, t)}) - \eta \barT^{(k, t)} \\
			& = \barX^{(k, t)} - \eta \beta E (\barX^{(k, t)}) + \eta Y_k,
		\end{aligned}
	\end{equation*}
	where $Y_k := \beta (E (\barX^{(k, t)}) - \barE^{(k, t)}) - \barT^{(k, t)}$ and $Y_k$ satisfies
	\begin{equation*}
		\norm{Y_k}\ff \leq \beta \norm{ E (\barX^{(k, t)}) - \barE^{(k, t)} }\ff 
		+ \norm{\barT^{(k, t)}}\ff
		\leq 1 + (2q + 1) M_G.
	\end{equation*}
	Then by straightforward calculations, we can obtain that
	\begin{equation*}
		\begin{aligned}
			& (\barX^{(k, t + 1)})\zz \barX^{(k, t + 1)} - I_p \\ 
			= {} & \dkh{\barX^{(k, t)} - \eta \beta E (\barX^{(k, t)}) + \eta Y_k}\zz
			\dkh{\barX^{(k, t)} - \eta \beta E (\barX^{(k, t)}) + \eta Y_k} - I_p \\
			= {} & (\barX^{(k, t)})\zz \barX^{(k, t)} - I_p 
			- 2 \eta \beta (\barX^{(k, t)})\zz \barX^{(k, t)} 
			\dkh{(\barX^{(k, t)})\zz \barX^{(k, t)} - I_p} \\
			& + \eta (\barX^{(k, t)})\zz Y_k + \eta^2 \beta^2 (\barX^{(k, t)})\zz \barX^{(k, t)} \dkh{(\barX^{(k, t)})\zz \barX^{(k, t)} - I_p}^2 \\
			& - \eta^2 \beta \dkh{(\barX^{(k, t)})\zz \barX^{(k, t)} - I_p} (\barX^{(k, t)})\zz Y_k 
			+ \eta Y_k\zz \barX^{(k, t)} \\
			& - \eta^2 \beta Y_k\zz \barX^{(k, t)} \dkh{(\barX^{(k, t)})\zz \barX^{(k, t)} - I_p} 
			+ \eta^2 Y_k\zz Y_k \\
			= {} & \dkh{I_n - \eta \beta (\barX^{(k, t)})\zz \barX^{(k, t)}}^2 
			\dkh{(\barX^{(k, t)})\zz \barX^{(k, t)} - I_p} \\
			& - \eta^2 \beta^2 (\barX^{(k, t)})\zz \barX^{(k, t)} 
			\dkh{(\barX^{(k, t)})\zz \barX^{(k, t)} - I_p} 
			+ \eta (\barX^{(k, t)})\zz Y_k + \eta Y_k\zz \barX^{(k, t)} \\
			& - \eta^2 \beta \dkh{(\barX^{(k, t)})\zz \barX^{(k, t)} - I_p} (\barX^{(k, t)})\zz Y_k \\
			& - \eta^2 \beta Y_k\zz \barX^{(k, t)} \dkh{(\barX^{(k, t)})\zz \barX^{(k, t)} - I_p}
			+ \eta^2 Y_k\zz Y_k.
		\end{aligned}
	\end{equation*}

	This further implies that
	\begin{equation*}
		\begin{aligned}
			& \norm{(\barX^{(k, t + 1)})\zz \barX^{(k, t + 1)} - I_p}\ff \\
			\leq {} & \norm{\dkh{I_n - \eta \beta (\barX^{(k, t)})\zz \barX^{(k, t)}}^2 
			\dkh{(\barX^{(k, t)})\zz \barX^{(k, t)} - I_p}}\ff \\
			& + \eta^2 \beta^2 \norm{ (\barX^{(k, t)})\zz \barX^{(k, t)} \dkh{(\barX^{(k, t)})\zz \barX^{(k, t)} - I_p} }\ff 
			+ 2 \eta \norm{(\barX^{(k, t)})\zz Y_k}\ff \\
			& + 2 \eta^2 \beta \norm{\dkh{(\barX^{(k, t)})\zz \barX^{(k, t)} - I_p} (\barX^{(k, t)})\zz Y_k}\ff
			+ \eta^2 \norm{Y_k\zz Y_k}\ff \\
			\leq {} & \dkh{1 - \dfrac{5}{6} \eta \beta}^2 
			\norm{(\barX^{(k, t)})\zz \barX^{(k, t)} - I_p}\ff 
			+ \dfrac{49}{216} \eta^2 \beta^2 + \dfrac{7}{3} (1 + (2q + 1) M_G) \eta \\
			& + \dfrac{7}{18} (1 + (2q + 1) M_G) \eta^2 \beta + (1 + (2q + 1) M_G)^2 \eta^2 \\
			\leq {} & \dkh{1 - \dfrac{5}{6} \eta \beta}^2 \norm{(\barX^{(k, t)})\zz (\barX^{(k, t)}) - I_p}\ff
			+ \eta \dkh{4 + 3 (2q + 1)M_G} + \eta^2 \dkh{1 + (2q + 1) M_G}^2, 
		\end{aligned}
	\end{equation*}
	where the last equality follows from the condition
	$\eta < \min\{\bar{\eta}_1, \bar{\eta}_2\}$.
	Now we consider the above relationship in the following two cases.
	
	\paragraph{Case I: $\norm{(\barX^{(k, t)})\zz \barX^{(k, t)} - I_p}\ff \leq 1/ 12$.} 
	Since $\eta < \bar{\eta}_3$,
	we have
	\begin{equation*}
		\norm{(\barX^{(k, t + 1)})\zz \barX^{(k, t + 1)} - I_p}\ff 
		\leq \norm{(\barX^{(k, t)})\zz \barX^{(k, t)} - I_p}\ff + \dfrac{1}{12} = \dfrac{1}{6}.
	\end{equation*}
	
	\paragraph{Case II: $\norm{(\barX^{(k, t)})\zz \barX^{(k, t)} - I_p}\ff > 1/ 12$.}
	It can be readily verified that
	\begin{equation*}
		\begin{aligned}
			& \norm{(\barX^{(k, t + 1)})\zz \barX^{(k, t + 1)} - I_p}\ff 
			- \norm{(\barX^{(k, t)})\zz \barX^{(k, t)} - I_p}\ff \\
			\leq {} & \dkh{\dkh{1 - \dfrac{5}{6} \eta \beta}^2 - 1}
			\norm{(\barX^{(k, t)})\zz \barX^{(k, t)} - I_p}\ff \\
			& + \eta \dkh{4 + 3 (2q + 1) M_G} + \eta^2 \dkh{1 + (2q + 1) M_G}^2 \\
			\leq {} & - \dfrac{5}{72} \eta \beta  + \eta \dkh{4 + 3 (2q + 1) M_G} 
			+ \eta^2 \dkh{1 + (2q + 1) M_G}^2,
		\end{aligned}
	\end{equation*}
	which together with $\beta > \underline{\beta}_2$
	and $\eta < \bar{\eta}_4$
	yields that
	\begin{equation*}
		\norm{(\barX^{(k, t + 1)})\zz \barX^{(k, t + 1)} - I_p}\ff 
		- \norm{(\barX^{(k, t)})\zz \barX^{(k, t)} - I_p}\ff \leq 0.
	\end{equation*}
	Hence, we arrive at 
	$\norm{(\barX^{(k, t + 1)})\zz \barX^{(k, t + 1)} - I_p}\ff 
	\leq \norm{(\barX^{(k, t)})\zz \barX^{(k, t)} - I_p}\ff \leq 1 / 6$.
	Combing the above two cases, we complete the proof.
\end{proof}

Lemma \ref{le:bound-X-inner} illustrates that, in the inner iteration of VRSGT,
the iterates are restricted in the bounded region $\cR$ if some conditions hold.
In fact, the same is true for the outer iteration of VRSGT.

\begin{lemma} \label{le:bound-X-outer}
	Let all the conditions in Lemma \ref{le:bound-X-inner} hold.
	Suppose $\barX^{(k, 0)} \in \cR$, 
	$\|T^{(k, 0)}_i\|\ff \leq (2q + 1) M_G$,
	$\norm{\bfX^{(k, 0)}}\ff \leq \sqrt{7dp / 6} + \sqrt{d}$,
	and $ \norm{\bar{\bfX}^{(k, 0)} - \bfX^{(k, 0)}}\ff \leq \sqrt{d} / (\beta^2 L_e)$.
	Then we have $\barX^{(k + 1, 1)} \in \cR$.
\end{lemma}

\begin{proof}
	The proof is similar to that of Lemma \ref{le:bound-X-inner},
	which is omitted here.
\end{proof}

\begin{proposition}\label{prop:bound}
	Suppose the conditions in Assumptions \ref{asp:network} and \ref{asp:smooth} hold.
	Let $\{\bfX^{(k, t)}\}$ be the iterate sequence 
	generated by Algorithm \ref{alg:VRSGT} with $X_{\mathrm{init}} \in \cR$
	and the stepsize $\eta$ and penalty parameter $\beta$ satisfy
	\begin{equation*}
		0 < \eta < \min\hkh{
			\bar{\eta}_1, \bar{\eta}_2, \bar{\eta}_3, \bar{\eta}_4,
			\bar{\eta}_5, \bar{\eta}_6
		},
		\mbox{~and~}
		\beta > \max \hkh{
			1, \underline{\beta}_2, \underline{\beta}_3, \underline{\beta}_4, \underline{\beta}_5
		},
	\end{equation*}
	respectively.
	Then for any $k \in \bN$ and $t \in [q + 1]$, it holds that 
	\begin{equation}\label{eq:bound-x}
		\begin{aligned}
		\barX^{(k, t)} \in \cR,~ 
		\norm{\avX^{(k, t)} - \bfX^{(k, t)}}\ff \leq \dfrac{\sqrt{d}}{\beta^2 L_E},~
		\norm{\bfX^{(k, t)}}\ff \leq \sqrt{\dfrac{7dp}{6}} + \sqrt{d}, \\
		\norm{\avD^{(k, t)} - \bfD^{(k, t)}}\ff \leq M_D,~
		\mbox{~and~} \norm{\bfD^{(k, t)}}\ff \leq C_D.
		\end{aligned}
	\end{equation}
\end{proposition}

	\begin{proof}
	We use mathematical induction to prove this proposition.
	The argument \eqref{eq:bound-x} directly holds 
	at $(\bfX^{(0, 0)}, \bfD^{(0, 0)})$ resulting from the initialization.
	To complete the proof, we consider the following two scenarios.
	
	\paragraph{Case I: Inner iterations.} 
	We assume that the argument \eqref{eq:bound-x} holds at $(\bfX^{(k, t)}, \bfD^{(k, t)})$,
	and investigate the situation at $(\bfX^{(k, t + 1)}, \bfD^{(k, t + 1)})$.
	
	Our first purpose is to show that 
	$\norm{\avX^{(k, t + 1)} - \bfX^{(k, t + 1)}}\ff \leq \sqrt{d} / (\beta^2 L_E)$.
	By straightforward calculations, we can attain that
	\begin{equation*}
		\begin{aligned}
			\norm{\avX^{(k, t + 1)} - \bfX^{(k, t + 1)}}\ff
			& = \norm{\dkh{\bfW - \bfJ} \dkh{\avX^{(k, t)} - \bfX^{(k, t)}} 
				+ \eta \dkh{\bfW - \bfJ} \bfD^{(k, t)}}\ff \\
			& \leq \norm{\dkh{\bfW - \bfJ} \dkh{\avX^{(k, t)} - \bfX^{(k, t)}}}\ff
			+ \eta \norm{\dkh{\bfW - \bfJ} \bfD^{(k, t)}}\ff \\
			& \leq \lambda \norm{\avX^{(k, t)} - \bfX^{(k, t)}}\ff
			+ \eta \lambda \norm{\bfD^{(k, t)}}\ff \\
			& \leq \lambda \norm{\avX^{(k, t)} - \bfX^{(k, t)}}\ff
			+ \eta \lambda C_D,
		\end{aligned}
	\end{equation*}
	which together with 
	$\eta < \bar{\eta}_5$
	implies that
	\begin{equation*}
		\norm{\avXkn - \bfXkn}\ff 
		\leq \dfrac{\sqrt{d}  \lambda}{\beta^2 L_E}
		+ \dfrac{\sqrt{d} \dkh{1 - \lambda}}{\beta^2 L_E}
		\leq \dfrac{\sqrt{d}}{\beta^2 L_E}.
	\end{equation*}
	
	Then we aim to prove that $\barX^{(k, t + 1)} \in \cR$ 
	and $\norm{\bfX^{(k, t + 1)}}\ff \leq \sqrt{7dp/6} + \sqrt{d}$.
	According to Lemmas \ref{le:bound-T} and \ref{le:bound-X-inner}, 
	we have $\barXkn \in \cR$.
	And it follows that
	\begin{equation*}
		\norm{\bfX^{(k, t + 1)}}\ff 
		\leq \norm{\avX^{(k, t + 1)}}\ff + \norm{\avX^{(k, t + 1)} - \bfX^{(k, t + 1)}}\ff
		\leq \sqrt{\dfrac{7dp}{6}} + \dfrac{\sqrt{d}}{\beta^2 L_E}
		\leq \sqrt{\dfrac{7dp}{6}} + \sqrt{d},
	\end{equation*}
	as a result of the condition $\beta > \underline{\beta}_3$.
	
	In order to finish the proof, we still have to show that
	$\norm{\avD^{(k, t + 1)} - \bfD^{(k, t + 1)}}\ff \leq M_D$
	and $\norm{\bfD^{(k, t + 1)}}\ff \leq C_D$.
	In fact, we have 
	\begin{equation*}
		\begin{aligned}
			\norm{S^{(k, t + 1)}_i - S^{(k, t)}_i}\ff 
			& \leq \dfrac{1}{\tau} \sum_{j \in \cC^{(k + 1, t)}} 
			\norm{H^{[j]}_{i} (X^{(k, t + 1)}_i) - H^{[j]}_{i} (X^{(k, t)}_i)}\ff \\
			& \leq (L_G + \beta L_E) \norm{X^{(k, t + 1)}_i - X^{(k, t)}_i}\ff, \quad \iid,
		\end{aligned}
	\end{equation*}
	which together with $\beta > \underline{\beta}_4$
	implies that
	\begin{equation*}
		\norm{\bfS^{(k, t + 1)} - \bfS^{(k, t)}}\ff \leq 2 \beta L_E \norm{\bfX^{(k, t + 1)} - \bfX^{(k, t)}}\ff.
	\end{equation*}
	Moreover, it can be readily verified that
	\begin{equation*}
		\begin{aligned}
			\norm{\bfX^{(k, t + 1)} - \bfX^{(k, t)}}\ff 
			& = \norm{\bfW \dkh{\bfX^{(k, t)} - \eta \bfD^{(k, t)}} - \bfX^{(k, t)}}\ff \\
			& = \norm{\dkh{I_{dn} - \bfW} \dkh{\avX^{(k, t)} - \bfX^{(k, t)}} 
				- \eta \bfW \bfD^{(k, t)}}\ff \\
			& \leq 2 \norm{\avX^{(k, t)} - \bfX^{(k, t)}}\ff + \eta \norm{\bfD^{(k, t)}}\ff \\
			& \leq \dfrac{2 \sqrt{d}}{\beta^2 L_E} + \eta C_D
			\leq \dfrac{3 \sqrt{d}}{\beta^2 L_E},
		\end{aligned}
	\end{equation*}
	where the last inequality follows from $\eta < \bar{\eta}_6$.
	Combing the above two relationships, we can obtain that
	\begin{equation*}
		\begin{aligned}
			& \norm{ \avD^{(k, t + 1)} - \bfD^{(k, t + 1)} }\ff \\
			= {} & \norm{ \dkh{\bfW - \bfJ} \dkh{\avD^{(k, t)} - \bfD^{(k, t)}} 
				- \dkh{I_{dn} - \bfJ} \dkh{\bfS^{(k, t + 1)} - \bfS^{(k, t)}} }\ff \\
			\leq {} & \norm{\dkh{\bfW - \bfJ} \dkh{\avD^{(k, t)} - \bfD^{(k, t)}}}\ff
			+ \norm{\dkh{I_{dn} - \bfJ} \dkh{\bfS^{(k, t + 1)} - \bfS^{(k, t)}}}\ff \\
			\leq {} & \lambda \norm{\avD^{(k, t)} - \bfD^{(k, t)}}\ff
			+ \norm{\bfS^{(k, t + 1)} - \bfS^{(k, t)}}\ff \\
			\leq {} & \lambda M_D + \dfrac{6 \sqrt{d}}{\beta}
			\leq M_D,
		\end{aligned}
	\end{equation*}
	where the last inequality follows from 
	$\beta > \underline{\beta}_5$.
	Therefore, we can obtain that
	\begin{equation*}
		\begin{aligned}
			\norm{\bfD^{(k, t + 1)}}\ff 
			& \leq \norm{\avD^{(k, t + 1)} - \bfD^{(k, t + 1)}}\ff + \norm{\avD^{(k, t + 1)}}\ff \\
			& = \norm{\avD^{(k, t + 1)} - \bfD^{(k, t + 1)}}\ff + \norm{\avS^{(k, t + 1)}}\ff \\
			& \leq M_D
			+ \sqrt{d} \dkh{(2q + 1) M_G + \beta M_E} 
			= C_D. 
		\end{aligned}
	\end{equation*}

	\paragraph{Case II: Outer iterations.} 
	Using the similar techniques,
	we can prove that the argument \eqref{eq:bound-x} holds at $(\bfX^{(k + 1, 1)}, \bfD^{(k + 1, 1)})$
	provided that it holds at $(\bfX^{(k, 0)}, \bfD^{(k, 0)})$.
	The corresponding proof is omitted here.

	Combining the above two cases, we complete the proof.
\end{proof}

Proposition \ref{prop:bound}
demonstrates that the sequence ${\bfX^{(k, t)}}$ is bounded
and the average of iterates $\barX^{(k, t)}$ is restricted in the region $\cR$.
Therefore, we can directly apply the existing convergence results of decentralized SVRG algorithms established in \cite{Sun2020improving,Xin2022fast},
although they impose the global Lipschitz smoothness property.
The following proposition is an implication of Theorem 1 in \cite{Sun2020improving}.

\begin{proposition} \label{prop:descent}
	Let the conditions in Assumptions \ref{asp:network} and \ref{asp:smooth} hold
	and $\{\bfX^{(k, t)}\}$ be the iterate sequence 
	generated by Algorithm \ref{alg:VRSGT} with $X_{\mathrm{init}} \in \cR$ and $t = q = \sqrt{l}$.
	Suppose the stepsize $\eta$ and penalty parameter $\beta$ satisfy
	\begin{equation*}
		0 < \eta < \min\hkh{
			\bar{\eta}_1, \bar{\eta}_2, \bar{\eta}_3, \bar{\eta}_4,
			\bar{\eta}_5, \bar{\eta}_6, \bar{\eta}_7, \bar{\eta}_8,
			\bar{\eta}_9
		},
		\mbox{~and~}
		\beta > \max \hkh{
			1, \underline{\beta}_2, \underline{\beta}_3, \underline{\beta}_4, \underline{\beta}_5
		},
	\end{equation*}
	respectively.
	Then it holds that
	\begin{equation*}
		\dfrac{1}{\sqrt{l} K} \sum_{k = 0}^K \sum_{t = 0}^q
		\bE \fkh{ \norm{\dfrac{1}{d} \sumiid H_i (X^{(k, t)}_i)}\fs
		+ \dfrac{1}{d} \sumiid \norm{X^{(k, t)}_i - \barX^{(k, t)}}\fs }
	\leq \dfrac{C_4}{\sqrt{l} K},
	\end{equation*}
	where $C_4 > 0$ is a constant defined in \eqref{eq:constants-2},
	and the expectation is taken over the randomness from the sampling step.
\end{proposition}

\begin{proof}[\bf Proof of Theorem \ref{thm:convergence}]
	By straightforward calculations, we have
	\begin{equation*}
		\begin{aligned}
			\norm{H (\barX^{(k, t)})}\fs
			& = \norm{\dfrac{1}{d} \sumiid H_i (\barX^{(k, t)})}\fs \\
			& = \norm{ \dfrac{1}{d} \sumiid H_i (\barX^{(k, t)}) - \dfrac{1}{d} \sumiid H_i (X^{(k, t)}_i) + \dfrac{1}{d} \sumiid H_i (X^{(k, t)}_i)}\fs \\
			& \leq 2 \norm{ \dfrac{1}{d} \sumiid H_i (\barX^{(k, t)}) - \dfrac{1}{d} \sumiid H_i (X^{(k, t)}_i) }\fs 
			+ 2 \norm{ \dfrac{1}{d} \sumiid H_i (X^{(k, t)}_i) }\fs \\
			& \leq \dfrac{2}{d} \sumiid \norm{ H_i (\barX^{(k, t)}) - H_i (X^{(k, t)}_i)}\fs
			+ 2 \norm{ \dfrac{1}{d} \sumiid H_i (X^{(k, t)}_i) }\fs \\ 
			& \leq \dfrac{2 \hat{L}^2}{d} \sumiid \norm{X^{(k, t)}_i - \barX^{(k, t)}}\fs
			+ 2 \norm{ \dfrac{1}{d} \sumiid H_i (X^{(k, t)}_i) }\fs,
		\end{aligned}
	\end{equation*}
	which implies that
	\begin{equation*}
		\begin{aligned}
			& \norm{H (\barX^{(k, t)})}\fs + \dfrac{1}{d} \sumiid \norm{X^{(k, t)}_i - \barX^{(k, t)}}\fs \\
			\leq {} & 2 \dkh{\hat{L}^2 + 1} \dkh{ 
				\norm{ \dfrac{1}{d} \sumiid H_i (X^{(k, t)}_i) }\fs
				+ \dfrac{1}{d} \sumiid \norm{X^{(k, t)}_i - \barX^{(k, t)}}\fs 
			}.
		\end{aligned}
	\end{equation*}
	According to Lemma \ref{le:expen}, it follows that
	\begin{equation*}
		\begin{aligned}
			\norm{H (\barX^{(k, t)})}\fs 
			& \geq \norm{G (\barX^{(k, t)})}\fs 
			+ \beta \norm{(\barX^{(k, t)})\zz \barX^{(k, t)} - I_p}\fs \\
			& \geq \norm{G (\barX^{(k, t)})}\fs 
			+ \norm{(\barX^{(k, t)})\zz \barX^{(k, t)} - I_p}\fs, 
		\end{aligned}
	\end{equation*}
	Combing the above two relationships, we can obtain that
	\begin{equation*}
		\begin{aligned}
			& \norm{\dfrac{1}{d} \sumiid G_i (\barX^{(k, t)})}\fs 
			+ \dfrac{1}{d} \sumiid \norm{X^{(k, t)}_i - \barX^{(k, t)}}\fs
			+ \norm{(\barX^{(k, t)})\zz \barX^{(k, t)} - I_p}\fs \\
			\leq {} & 2 \dkh{\hat{L}^2 + 1} \dkh{ 
				\norm{ \dfrac{1}{d} \sumiid H_i (X^{(k, t)}_i) }\fs
				+ \dfrac{1}{d} \sumiid \norm{X^{(k, t)}_i - \barX^{(k, t)}}\fs 
			}.
	\end{aligned}
	\end{equation*}
	This together with Proposition \ref{prop:descent} yields that
	\begin{equation*}
		\dfrac{1}{\sqrt{l} K} \sum_{k = 0}^K \sum_{t = 0}^q
		\bE \fkh{\mathtt{StaGap} (\bfX^{(k, t)})}
		\leq \dfrac{2 (\hat{L}^2 + 1) C_4}{\sqrt{l} K}
		= \dfrac{C}{\sqrt{l} K}.
	\end{equation*}
	Since
	\begin{equation*}
		\dfrac{1}{\sqrt{l} K} \sum_{k = 0}^K \sum_{t = 0}^q
		\bE \fkh{\mathtt{StaGap} (\bfX^{(k, t)})}
		\geq \min_{\substack{k = 0, 1, \dotsc, K \\ t = 0, 1, \dotsc, q}} 
		\bE \fkh{ \mathtt{StaGap} (\bfX^{(k, t)}) },
	\end{equation*}
	we complete the proof.
\end{proof}

\end{appendices}

\end{document}